\documentclass[reqno]{article}
\usepackage{amssymb}
\usepackage{amsmath}
\usepackage{amsthm}
\usepackage{mathtools}
\usepackage{graphicx}
\usepackage{comment}

\usepackage{amssymb, mathrsfs, amsfonts, amsmath}

\usepackage[colorlinks]{hyperref}

\usepackage{tikz}
\usetikzlibrary{arrows}

\usepackage[english]{babel}

\usepackage[margin=1.4in, centering]{geometry}

\DeclareSymbolFont{bbold}{U}{bbold}{m}{n}
\DeclareSymbolFontAlphabet{\mathbbold}{bbold}
\newtheorem{thm}{Theorem}%[Section]
\newtheorem{prop}[thm]{Proposition}

\newtheorem{cor}[thm]{Corollary}
\theoremstyle{definition}

\theoremstyle{remark}
\newtheorem{rem}[thm]{Remark}
\newcommand\E[1]{}

\newcommand{\R}{\mathbb{R}}

\newcommand{\N}{\mathbb{N}}
\newcommand{\Z}{\mathbb{Z}}

\def\XXint#1#2#3{{\setbox0=\hbox{$#1{#2#3}{\int}$ }
\vcenter{\hbox{$#2#3$ }}\kern-.6\wd0}}

\title{Bounds for lacunary bilinear spherical and triangle maximal functions}
\author{Tainara Borges\thanks{Department of Mathematics, Brown University, Providence, RI 02912, tainara\_gobetti\_borges@brown.edu\\ ORCID: 0000-0002-3016-6138}\and Benjamin Foster\thanks{Department of Mathematics, Stanford University, Stanford, CA 94305, bfost@stanford.edu \\ ORCID: 0009-0003-8265-6790 }}

\begin{document}

\maketitle

\begin{abstract}
We prove $L^p(\R^d)\times L^q(\R^d)\rightarrow L^r(\R^d)$ bounds for certain lacunary bilinear maximal averaging operators with parameters satisfying the H\"older relation $1/p+1/q=1/r$. The boundedness region that we get contains at least the interior of the H\"older boundedness region of the associated single scale bilinear averaging operator. In the case of the lacunary bilinear spherical maximal function in $d\geq 2$, we prove boundedness for any $p,q\in (1,\infty]^2$, which is sharp up to boundary; we then show how to extend this result to a more degenerate family of surfaces where some curvatures are allowed to vanish. For the lacunary triangle averaging maximal operator, we have results in $d\geq 7$, and the description of the sharp region will depend on a sharp description of the H\"older bounds for the single scale triangle averaging operator, which is still open.
\end{abstract}

\vspace{1cm}

\noindent\textbf{Acknowledgements.} The first author would like to thank Professor Jill Pipher for her support and for helpful discussions throughout the process of writing this paper. We also would like to thank Professor Yumeng Ou for introducing us to some of the questions in this paper. \vspace{0.5cm} \\ \\
\textbf{Keywords.} Bilinear averaging operator, coarea formula, high-low decomposition, lacunary maximal function, spherical average

\newpage

\section{Introduction and main results}

Given a compactly supported finite Borel measure $\sigma$ in $\R^d$ an important class of maximal operators is given by 
$$M_{\sigma}(f)(x)=\sup_{t>0}\left|\int f(x-tz)d\sigma(z)\right|.$$
In the particular case of $\sigma=\sigma_{d-1}$, the normalized surface measure in the unit sphere $S^{d-1}$, one gets the spherical maximal function $M_{\sigma_{d-1}}$, which we will denote by $\mathcal{S}$. The study of the boundedness properties of $\mathcal{S}$ were initiated by Stein \cite{stein}. He showed that $\mathcal{S}$ is bounded in $L^{p}$ if $d\geq 3$ and $\frac{d}{d-1}< p\leq \infty$ and that boundedness fails for $d\geq 2$ and $p\leq \frac{d}{d-1}$. The remaining case $d=2$ was proven later by Bourgain \cite{bourgaind2}; this case was more complicated since the operator fails to be bounded in $L^2(\R^2)$, requiring more delicate Fourier-analytic tools.

More general maximal functions are obtained when the supremum is taken over dyadic dilates of a subset $E\subset [1,2]$. In the case of $S=S^{d-1}$ and $E=\{1\}$, one ends up with the so called lacunary spherical maximal function
$$\mathcal{S}_{lac}(f)(x)=\sup_{l\in \Z} \left|\int_{S^{d-1}}f(x-2^{-
l}z)d\sigma_{d-1} (z)\right|.$$
The operator $\mathcal{S}_{lac}$ has better boundedness properties than the full spherical maximal function $\mathcal{S}$; it was shown by C. Calderón \cite{calderon} and Coifman and Weiss \cite{coifmanweiss} that $\mathcal{S}_{lac}$ is bounded in $L^p$ for any $1<p\leq \infty$. Some generalizations of these facts for more general dilation sets $E$ can be found in \cite{SeegerWainger} and \cite{DuoanVargas}. One of the interesting consequences of the results in \cite{DuoanVargas} is that for a compactly supported finite Borel measure $\sigma$ in $\R^d$, as long as its Fourier transform $\hat{\sigma}(\xi)$ has some decay $|\hat{\sigma}(\xi)|\leq C|\xi|^{-a}$, with $a>0$ then $$M_{\sigma,lac}(f)(x)=\sup_{l\in\Z}\left|\int f(x-2^{-l}z)d\sigma(z)\right|$$
is bounded in $L^{p}$, for all $p>1$. Also, if $|\hat{\sigma}(\xi)|\leq C|\xi|^{-\frac{d-1}{2}}$ and $d\geq 3$ then $M_{\sigma}$ is bounded in $L^{p}$ for all $p>\frac{d-1}{d}$, recovering Stein's result.

We are interested in proving some bilinear analogues of these results in the lacunary setting. The bilinear averages we are interested in are given by bilinear convolutions with a compactly supported finite Borel measure $\mu$ in $\R^{2d}$. Namely, for $t>0$, and $f,g\in C^{\infty}_0(\R^d)$, define the bilinear average at scale $t>0$ as
\begin{equation}
    \mathcal{A}_{\mu,t}(f,g)(x)=\int_{S^{2d-1}}f(x-ty)g(x-tz)\,d\mu(y,z),\,x\in \R^d.
\end{equation}
The associated lacunary spherical maximal function is then given by
\begin{equation}
    \mathcal{M}_{\mu,lac}(f,g)(x)=\sup_{l \in \Z}|\mathcal{A}_{\mu,2^{-l}}(f,g)(x)|.
\end{equation}

In the particular case of $\mu=\sigma_{2d-1}$ we might omit the measure $\mu$ in the notation. By allowing any parameter $t>0$ in the supremum above one gets the (full) bilinear spherical maximal function
\begin{equation}
    \mathcal{M}(f,g)(x)=\sup_{t>0} |\mathcal{A}_{t}(f,g)(x)|.
\end{equation}

For $d\geq 2$, the study of the Lebesgue bounds for $\mathcal{M}$ 
has attracted a lot of attention in recent years. In particular, people have studied the boundedness region of the operator, that is, the description of all the parameters $1\leq p,q\leq \infty $ and $r>0$ such that $\mathcal{M}:L^p\times L^q\rightarrow L^r$ is bounded. For both $\mathcal{M}$ and $\mathcal{M}_{lac}$, it is necessary to impose the H\"older relation $p^{-1}+q^{-1}=r^{-1}$ because of scaling. The operator $\mathcal{M}$ first appeared in \cite{GGIS}, and subsequently, its Lebesgue bounds were  studied by multiple authors (\cite{BGHHO},\cite{GHH},\cite{HHY}) until the description of the sharp boundedness region for $\mathcal{M}$ in $d\geq 2$ was finally settled by \cite{JL} (see Theorem  \ref{boundsjl} below for the statement of these sharp bounds). In that paper, they exploited an interesting coarea formula which allowed them to slice the integral in $S^{2d-1}$ into integrals over lower dimensional spheres.

Even though the question of the boundedness region for $\mathcal{M}$ was settled in \cite{JL}, the sharp boundedness region for the lacunary case in $d\geq 2$ was still open and did not follow from the slicing strategy exploited in their paper. Indeed, when studying the boundedness properties of the bilinear spherical maximal function $\mathcal{M}$, Jeong and Lee \cite{JL} showed that if $F$ is a continuous function defined in $\R^{2d}$ and $(y,z)\in \R^d\times \R^d$, then the following slicing formula holds
$$\int_{S^{2d-1}}F(y,z)d\sigma (y,z)=\int_{B^{d}(0,1)}\int_{S^{d-1}} F(y,\sqrt{1-|y|^2}z)d\sigma_{d-1} (z) (1-|y|^2)^{\frac{d-2}{2}}\,dy.$$

In particular, for any $f,g\in C^{\infty}_{0}(\R^d)$,
\begin{equation}\label{slicing formula}
    \begin{split}
     \mathcal{A}_t(f,g)(x)=&\int_{B^{d}(0,1)}f(x-ty)\int_{S^{d-1}} g(x-t\sqrt{1-|y|^2}z)d\sigma_{d-1} (z) (1-|y|^2)^{\frac{d-2}{2}}\,dy \\
     =&\int_{B^{d}(0,1)}f(x-ty)A_{t(\sqrt{1-|y|^2})}(g)(x)(1-|y|^2)^{\frac{d-2}{2}}\,dy ,  
    \end{split}
\end{equation}
where $A_t f(x):=\int_{S^{d-1}}f(x-ty)d\sigma(y)$ is the linear spherical average at scale $t$, and they used that to show $$\mathcal{M}(f,g)(x)\lesssim Mf(x)\mathcal{S}g(x),$$
where $M$ is the Hardy-Littlewood maximal function. If we restrict ourselves to $t\in 2^{\mathbb{Z}}$ in the hope of controlling $\mathcal{M}_{lac}(f,g)$ by the product of $Mf(x)$ with the lacunary (sub)linear spherical maximal function $\mathcal{S}_{lac}$, we quickly run into trouble. Even when $t= 2^{-l}$ with $l\in \Z$, one can observe that in the slicing formula \eqref{slicing formula}, one needs to take into account spherical averages of $g$ in all scales $2^{-l}\sqrt{1-|y|^2}$ for any $|y|<1$. As a result their slicing strategy does not immediately imply a larger boundedness region for the lacunary bilinear maximal function than the one they prove for the full bilinear spherical maximal function, and we must instead turn to other methods.

Observe that the definition of $\mathcal{A}_t$ makes sense for any $d\geq 1$. In \cite{MCZZ} the authors studied $\mathcal{M}$ and $\mathcal{M}_{lac}$ in the case $d=1$ and they proved the sharp bounds up to the boundary. The sharp boundedness region for $\mathcal{M}$ in $d=1$ was also obtained independently by \cite{dosidisramos}. 

In this article, we almost entirely settle the question of Lebesgue bounds for $\mathcal{M}_{lac}$ in $d\geq 2$, providing the lacunary counterpart of the boundedness results for $\mathcal{M}$ in \cite{JL} and the bilinear analogue of the bounds for $\mathcal{S}_{lac}$ in \cite{calderon} and \cite{coifmanweiss}.

\begin{thm}[Bounds for the lacunary bilinear spherical maximal function]\label{boundsspherical} Assume $d\geq 2$. Let $p,q\in (1,\infty]$ and $r\in (0, \infty)$ satisfy the H\"older relation $\frac{1}{r}=\frac{1}{p}+\frac{1}{q}$. Then
    \begin{equation}
     \|\mathcal{M}_{lac}(f,g)\|_r= \|\sup_{l \in \Z}|\mathcal{A}_{2^{-l}}(f,g)|\|_r\lesssim  \|f\|_p\|g\|_q.
    \end{equation}
    Moreover, $\mathcal{M}_{lac}$ does not satisfy strong bounds $L^1\times L^{\infty}\rightarrow L^{1}$ or $L^{\infty}\times L^1\rightarrow L^1$, but it satisfies weak bounds $\mathcal{M}_{lac}:L^1\times L^{\infty}\rightarrow L^{1,\infty}$ and $\mathcal{M}_{lac}:L^{\infty}\times L^{1}\rightarrow L^{1,\infty}$.
\end{thm}

    \begin{rem}
    Bounds along certain segments of the lines $p=1$ or $q=1$ follow from the observation that $\mathcal{M}_{lac}(f,g)(x)\leq \mathcal{M}(f,g)(x) $  and the known bounds from \cite{JL} (see Figure \ref{figure1}). For example, for $p=1$, let $q\in (\frac{d}{d-1}, \infty)$ and $1/r=1/p+1/q$, then
        $$\|\mathcal{M}_{lac}(f,g)\|_{r}\lesssim \|f\|_1 \|g\|_{q}.$$
        We conjecture that bound above holds at least for $q\in (1,\infty)$, but the methods in this paper did not extend to that range.
    \end{rem} 

    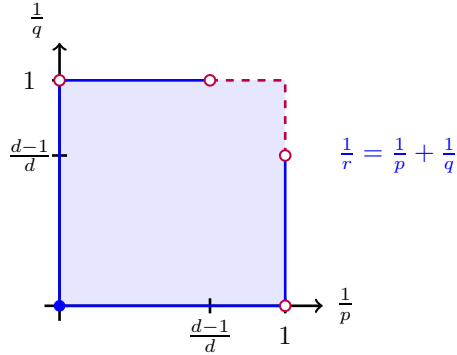
\begin{figure}[h]\label{figure1}
\begin{center}
         \scalebox{1}{
\begin{tikzpicture}
\fill[blue!10!white] (0,0)--(3,0)--(3,3)--(0,3)--(0,0);

\draw (-0.3,3.8) node {$\frac{1}{q}$};
\draw (3.8,0) node {$\frac{1}{p}$};
\draw (3,-0.4) node {$1$};
\draw (-0.4,3) node {$1$};
\draw (-0.4,2) node{$\frac{d-1}{d}$};
\draw (2,-0.4) node{$\frac{d-1}{d}$};

\draw[->,line width=1pt] (-0.2,0)--(3.5,0);
\draw[->,line width=1pt] (0,-0.2)--(0,3.5);

\draw[-,line width=1pt] (3,-0.1)--(3,0.1);
\draw[-,line width=1pt] (-0.1,3)--(0.1,3);
\draw[-,line width=1pt] (-0.1,2)--(0.1,2);
\draw[-,line width=1pt] (2,-0.1)--(2,0.1);

\draw[line width=1pt,blue,line width=1pt] (0,0)--(3,0)--(3,2);
\draw[line width=1pt,blue,line width=1pt] (0,0)--(0,3)--(2,3);
\draw[dashed,purple,line width=1pt] (2,3)--(3,3)--(3,2);

\draw[blue] (4.5,2) node {$\frac{1}{r}=\frac{1}{p}+\frac{1}{q}$};

\fill[white] (3,0) circle (2pt);
\fill[white] (0,3) circle (2pt);
\fill[white] (3,2) circle (2pt);
\fill[white] (2,3) circle (2pt);
\filldraw[blue] (0,0) circle (2pt);
\draw[purple,line width=0.75pt] (3,0) circle (2pt);
\draw[purple,line width=0.75pt] (0,3) circle (2pt);
\draw[purple,line width=0.75pt] (3,2) circle (2pt);
\draw[purple,line width=0.75pt] (2,3) circle (2pt);

\end{tikzpicture}}
\end{center}
\caption{H\"older bounds for the lacunary bilinear spherical maximal function in $d\geq 2$.}
\end{figure}

\begin{rem}
    $\mathcal{S}_{lac}$ fails to be bounded in $L^1(\R^d)$, and there are some interesting results about the endpoint behavior of $\mathcal{S}_{lac}$ in the literature (see \cite{cladek2017} and references therein). One open question is what kind of bounds can be proved for $\mathcal{M}_{lac}$ in the missing pieces of the boundary and if any of the linear methods in these papers can be adapted to our bilinear setting.
\end{rem}

The strategy and methods used in the proof of Theorem \ref{boundsspherical} actually generalize to take care of a more general class of operators including the bilinear averaging operators which appeared in the work of Lee and Shuin \cite{leeshuin}. We refer to Section \ref{leeshuinsection} for the definition of the region $\mathcal{R}^{\mathbf{a}}$ and the bilinear averaging operators $\mathcal{A}_{t}^{\mathbf{a}}(f,g)$, for $\mathbf{a}\in [1,\infty)^2$.

\begin{thm}[Bounds for the lacunary maximal operators associated to degenerate surfaces]\label{lacunaryleeshuin}
   Assume $d\geq 2$ and let $\mathbf{a}=(\mathbf{a_1},\mathbf{a_2})\in [1,d]^2$.  Then, for any $(\frac{1}{p},\frac{1}{q})\in \text{int}{(\mathcal{R}^{\mathbf{a}})}$, and $r\in (0, \infty)$ given by the H\"older relation $\frac{1}{r}=\frac{1}{p}+\frac{1}{q}$, it holds that
    \begin{equation}
\|\sup_{l\in\Z}|\mathcal{A}_{2^{-l}}^{\mathbf{a}}(f,g)|\|_{r}\lesssim  \|f\|_p\|g\|_q.
    \end{equation}
\end{thm}

Another interesting bilinear averaging operator that has been studied in the last few years 
(\cite{triangleaveraging}, \cite{IPS}, \cite{sparsetriangle}) is the triangle averaging operator given at scale $t>0$ by
\begin{equation}
    \mathcal{T}_{t}(f,g)(x)=\int_{\mathcal{I}} f(x-ty)g(x-tz)d\mu(y,z),
\end{equation}
where $\mu$ is the natural surface measure on the submanifold %of $\R^{2d}$ given by 
$$\mathcal{I}=\{(y,z)\in \R^{2d}\colon |y|=|z|=|y-z|=1 \}.$$

By fixing an equilateral triangle in $\R^d$, that is, $u,v\in \R^d$ with $|u|=|v|=|u-v|=1$, one can also write 
\begin{equation}
    \mathcal{T}_{t}(f,g)(x)=\int_{O(d)} f(x-tRu)g(x-tRv)d\mu(R),
\end{equation}
where $\mu$ is the normalized Haar measure on the group $O(d)$, the orthogonal group in $\R^d$, a fact which was already exploited in \cite{IPS}, for example. The study of the triangle averaging operator is closely related to Falconer type results for three point configurations of points in compact sets of $\R^d$ with large enough Hausdorff dimension (\cite{threepointconf}, \cite{GGIP}, \cite{IL}).

In this paper, we give a partial description of the Lebesgue boundedness region of the associated lacunary bilinear triangle averaging operator
\begin{equation}
    \mathcal{T}_{lac}(f,g)(x)=\sup_{l\in \Z} |\mathcal{T}_{2^l}(f,g)(x)|.
\end{equation}

Such region will be closely related to parameters for which one has H\"older bounds for the bilinear triangle averaging operator $\mathcal{T}_1$. Let $d\geq 2$. Denote by $\mathcal{R}_1$ the region inside $[0,1]\times [0,1]$ given by the convex closure of the points 
\begin{equation}\label{regionr1}
    (0,0),\,(0,1),\,(1,0)\text{ and }\left(\frac{d}{d+1},\frac{d}{d+1}\right).
\end{equation}

\begin{thm}[Bounds for the lacunary triangle averaging maximal function]\label{boundstriangle}
   Assume $d\geq 7$. For any $(\frac{1}{p},\frac{1}{q})\in \text{int}{(\mathcal{R}_1)}$, and $r\in (0, \infty)$ given by the H\"older relation $\frac{1}{r}=\frac{1}{p}+\frac{1}{q}$, it holds that
    \begin{equation}
\|\mathcal{T}_{lac}(f,g)\|_{r}\lesssim  \|f\|_p\|g\|_q.
    \end{equation}
    Moreover, for any $d\geq 2$ boundedness holds for $(1/p,1/q)$ in the half open segment connecting $(0,0)$ to $(0,1)$, or in the half open segment connecting $(0,0)$ to $(1,0)$. That is, for any $q\in (1,\infty]$, we have 
$$ \|\mathcal{T}_{lac}(f,g)\|_q\lesssim \|f\|_{\infty} \|g\|_q$$
and 
$$\|\mathcal{T}_{lac}(f,g)\|_q\lesssim \|f\|_q \|g\|_{\infty}.$$
    
\end{thm}

\begin{rem}
    One can also state the theorem above in terms of $(1/p,1/q,1/r)$ in the interior of the H\"older boundedness region of $\mathcal{T}_1$, namely 
    $$\mathcal{B}_{\mathcal{T}_1}=:\{(1/p,1/q,1/r)\colon (p,q,r)\in [1,\infty]^2\times (0,\infty],\,1/p+1/q=1/r\,\text{ and } \|\mathcal{T}_1\|_{L^p\times L^q\rightarrow L^r}<\infty\}.$$
    One difficulty is that the sharp description of $\mathcal{B}_{\mathcal{T}_1}$ is not known, and all we know so far is that it contains the H\"older exponents $(1/p,1/q,1/r)$ with $(1/p,1/q)$ in the quadrilateral $\mathcal{R}_1$. The main reason why we have to restrict ourselves to $d\geq 7$ in the theorem above is due to the fact that in Proposition \ref{PS ij estimate} we are only able to prove decay bounds at the exponent $L^2\times L^2\rightarrow L^1$ when $d\geq 7$.
\end{rem}

There are several common features to deducing the desired bounds for the lacunary bilinear maximal operators. We generally follow a scheme of extending boundedness estimates for single scale averages to the lacunary maximal functions by proving decay bounds for pieces of the single averaging operator in which we localize the Fourier support in both variables separately.

 The passage from the decay bounds for pieces of the single scale averaging operators $\mathcal{A}_{\mu,1}$ to bounds for the lacunary bilinear maximal functions can be seen as an adaptation of some ideas in \cite{HHY} where, inspired by \cite{DuoanVargas} in the linear case, they were able to prove bounds for the multi-scale operator $\mathcal{M}(f,g)(x)=\sup_{t>0}|\mathcal{A}_t(f,g)(x)|$ by proving decay bounds for frequency localized pieces of the single scale operator $\tilde{\mathcal{M}}(f,g)(x)=\sup_{t\in [1,2]} |\mathcal{A}_t(f,g)|$.
 
In the proof of the decay bounds for frequency localized pieces of $\mathcal{A}_1$ and more generally $\mathcal{A}^{\mathbf{a}}_1$, one of the main innovations is that we combine the coarea formula used in \cite{JL} with polar coordinates in the ball to write in the case of $\mathcal{A}_1$  
\begin{equation}
 \mathcal{A}_{1}(f,g)(x)=\int_{0}^{1} \lambda^{d-1}(1-\lambda^2)^{\frac{d-2}{2}}A_{\lambda}f(x)A_{\sqrt{1-\lambda^2}}g(x) d\lambda,  
\end{equation}
and this equality can be used to get decay bounds for the bilinear averages from the well understood decay bounds for linear spherical averages. A similar strategy is used for $\mathcal{A}_{1}^{\mathbf{a}}$.

\textbf{Plan of the article.} We first introduce in Section \ref{preliminaries} some preliminary facts that will be useful along the paper. In Section \ref{sectiondecayforpieces}, we describe how to break down the bilinear spherical averaging operator into pieces and we obtain decay estimates for those pieces. Those estimates will be key in the proof of Theorem \ref{boundsspherical}, which is given in Section \ref{prooffirstthm}. In Section \ref{proofsecondthm}, we prove Theorem \ref{boundstriangle}. In that section we will get decay estimates for the pieces of the triangle averaging operator by making use of a bilinear multiplier boundedness criteria that we recall in Section \ref{preliminaries}. In Section \ref{leeshuinsection}, we show how the methods we used to prove the decay of the pieces of the bilinear spherical maximal function can be generalized for a more general class of bilinear averaging operators, as stated in Theorem \ref{lacunaryleeshuin}.\\

\textbf{Notation.}
Throughout this paper for $p\in (0,\infty]$, $\|\cdot\|_p$ stands for the usual norm in $L^p(\R^d)$, namely 
$$\|f\|_p=\left(\int_{\R^d} |f(x)|^p dx\right)^{1/p}, \,\text{if } p<\infty.$$
When $0<p<1$, even though $\|\cdot\|_p$ is not a norm, $\|\cdot\|_p^p$ satisfies the triangle inequality, which will be sufficient for the applications we need. For two quantities $A,B\geq 0$, $A\lesssim B$ means that there exists $C>0$ such that $A\leq C B$. For a measurable set $E\subset \R^d$, we let $\chi_E$ denote the indicator function of the set $E$. In addition to the previously defined maximal operators, we will denote by $Mf$ the Hardy-Littlewood maximal function function of $f$. The linear spherical average at scale $t$ is denoted by $A_t$. We will use $d\sigma_k$ to denote the natural surface measure on $S^{k}$. When $k=2d-1$, we drop the subscript and write only $d\sigma$. We will use $d\mu$ to denote the natural surface measure on the triangle manifold $\mathcal{I}$.

\section{Preliminary facts}\label{preliminaries}

\subsection{Relevant facts about \texorpdfstring{$\mathcal{A}_1$}{Lg}}

\begin{thm}[\cite{IPS}] Let $d\geq 2$, then
\begin{equation}\label{boundA1}
    \mathcal{A}_1:L^1(\R^d)\times L^1(\R^d)\rightarrow L^{1/2}(\R^d) \text{ is bounded.}
\end{equation}
\end{thm}

For $p,q,r\in [1,\infty]$ such that $\frac{1}{r}=\frac{1}{p}+\frac{1}{q}$ (that is, we are in the Banach case when $r\geq 1$), it is immediate from Minkowski's inequality combined with H\"older's inequality that 
$$\mathcal{A}_1:L^{p}(\R^d)\times L^{q}(\R^d)\rightarrow L^{r}(\R^d)\text{ is bounded.}$$

Interpolating the bound in (\ref{boundA1}) with the trivial bounds above, one gets the following corollary.

\begin{cor} \label{holderboundsforA1}Let $d\geq 2$. Then for all $1\leq p,q \leq \infty$ and $r\in (0,\infty]$ given by the H\"older relation  $1/r=1/p+1/q$, one has that
\begin{equation}
    \mathcal{A}_1:L^{p}(\R^d)\times L^{q}(\R^d)\rightarrow L^r(\R^d) \text{ is bounded.}
\end{equation}

\end{cor}

\begin{rem}
Due to the single scale nature of $\mathcal{A}_1$, one can also prove $L^p$ improving bounds for $\mathcal{A}_1$, that is, boundedness from $L^p\times L^q$ into $L^r$ for $1/r<1/p+1/q$. However, the sharp region of parameters for which this holds is still not known. So far, all it is known is that it contains the known boundedness region of the localized bilinear maximal function $\tilde{\mathcal{M}}=\sup_{t\in [1,2]}|\mathcal{A}_{t}(\cdot, \cdot)|$ (see \cite{JL} or \cite{BFOPZ} for the description of that region), and also the points $(1,1,1)$ and $(1,1,1/2)$ \cite{IPS}. For $d=1$ there is a better understanding of the $L^p$ improving region given by the results in \cite{SS},\cite{DOberlin}, and \cite{BS}. For the purposes of this article we will not need $L^p$ improving estimates and Corollary \ref{holderboundsforA1} will be enough.

\end{rem}

The sharp boundedness region for the full bilinear maximal function in $d\geq 2$ is known due to Jeong and Lee \cite{JL}:
\begin{thm}[\cite{JL}]\label{boundsjl} Let $d\geq 2$, $p,q\in [1,\infty]$ and $r\in (0,\infty]$. Then the estimate 
\begin{equation}
   \|\mathcal{M}(f,g)\|_r\lesssim \|f\|_{p}\|g\|_q
\end{equation}
holds if and only if $1/r=1/p+1/q$ and $1/r<(2d-1)/d$, except for the case $(p,q,r)=(1,\infty, 1)$ and $(p,q,r)=(\infty,1,1)$ where the boundedness fails. Also one has weak type bounds
$$\|\mathcal{M}(f,g)\|_{L^{1,\infty}}\lesssim \|f\|_1\|g\|_{\infty}$$
and 
$$\|\mathcal{M}(f,g)\|_{L^{1,\infty}}\lesssim \|f\|_{\infty}\|g\|_{1}.$$
    
\end{thm}

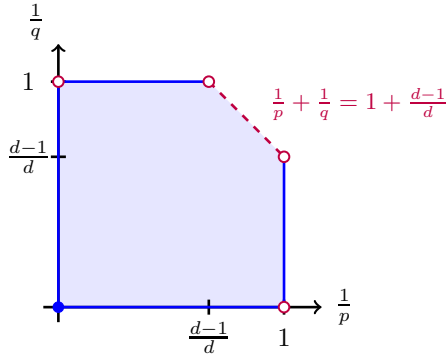
\begin{figure}[h]
\begin{center}
         \scalebox{1}{
\begin{tikzpicture}
\fill[blue!10!white] (0,0)--(3,0)--(3,2)--(2,3)--(0,3)--(0,0);

\draw (-0.3,3.8) node {$\frac{1}{q}$};
\draw (3.8,0) node {$\frac{1}{p}$};
\draw (3,-0.4) node {$1$};
\draw (-0.4,3) node {$1$};
\draw (-0.4,2) node{$\frac{d-1}{d}$};
\draw (2,-0.4) node{$\frac{d-1}{d}$};

\draw[->,line width=1pt] (-0.2,0)--(3.5,0);
\draw[->,line width=1pt] (0,-0.2)--(0,3.5);
\draw[-,line width=1pt] (3,-0.1)--(3,0.1);
\draw[-,line width=1pt] (-0.1,3)--(0.1,3);
\draw[-,line width=1pt] (-0.1,2)--(0.1,2);
\draw[-,line width=1pt] (2,-0.1)--(2,0.1);

\draw[line width=1pt,blue,line width=1pt] (0,0)--(3,0)--(3,2);
\draw[line width=1pt,blue,line width=1pt] (2,3)--(0,3)--(0,0);
\draw[dashed,purple,line width=1pt] (2,3)--(3,2);

\draw[purple] (4,2.7) node {\small{$\frac{1}{p}+\frac{1}{q}=1+\frac{d-1}{d}$}};

\fill[white] (3,0) circle (2pt);
\fill[white] (0,3) circle (2pt);
\fill[white] (3,2) circle (2pt);
\fill[white] (2,3) circle (2pt);
\filldraw[blue] (0,0) circle (2pt);
\draw[purple,line width=0.75pt] (3,0) circle (2pt);
\draw[purple,line width=0.75pt] (0,3) circle (2pt);
\draw[purple,line width=0.75pt] (3,2) circle (2pt);
\draw[purple,line width=0.75pt] (2,3) circle (2pt);

\end{tikzpicture}}
\end{center}
\caption{Jeong and Lee boundedness region for the bilinear spherical maximal function when $d\geq 2$.}
\end{figure}

\subsection{Relevant facts about \texorpdfstring{$\mathcal{T}_1$}{Lg}}

\begin{thm}{\cite{IPS}}
The triangle averaging operator satisfies the bound
$$\mathcal{T}_1:L^{\frac{d+1}{d}}(\R^d)\times L^{\frac{d+1}{d}}(\R^d) \rightarrow L^{s}(\R^d), \text{ for all }s\in [(d+1)/2d,1],\,d\geq 2.$$
Moreover, when the target space is $L^1(\R^d)$ one has a sharp description of the boundedness region, namely
$$\mathcal{T}_1: L^{p}(\R^d)\times L^{q}(\R^d)\rightarrow L^{1}(\R^d)$$
if and only if $(\frac{1}{p},\frac{1}{q})$ lies in the convex hull of the points ${(0,1),(1,0)}$ and $(\frac{d}{d+1},\frac{d}{d+1})$.
\end{thm}
 
Again, for H\"older exponents $1/r=1/p+1/q$, with $r\geq 1$, the boundedness of $\mathcal{T}_1$ follows directly from Minkowski's inequality for integrals and H\"older's inequality. Interpolating this with the bound $\mathcal{T}_1:L^{\frac{d+1}{d}}\times L^{\frac{d+1}{d}}\rightarrow L^{\frac{d+1}{2d}} $ above, one has the following corollary, where $\mathcal{R}_1$ is the region defined in (\ref{regionr1}).

\begin{cor} \label{holderboundsforT1}Let $d\geq 2$. Then for all $1\leq p,q \leq \infty$ and $r\in (0,\infty]$ given by the H\"older relation  $1/r=1/p+1/q$, and $(1/p,1/q)\in \mathcal{R}_1$, one has that 
\begin{equation}
    \mathcal{T}_1:L^{p}(\R^d)\times L^{q}(\R^d)\rightarrow L^r(\R^d) \text{ is bounded.}
\end{equation}
\end{cor}

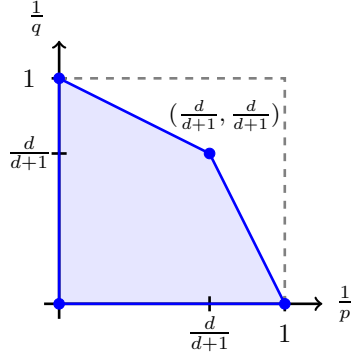
\begin{figure}[h]
\begin{center}
         \scalebox{1}{
\begin{tikzpicture}
\fill[blue!10!white] (0,0)--(3,0)--(2,2)--(0,3)--(0,0);

\draw (-0.3,3.8) node {$\frac{1}{q}$};
\draw (3.8,0) node {$\frac{1}{p}$};
\draw (3,-0.4) node {$1$};
\draw (-0.4,3) node {$1$};
\draw (-0.4,2) node{$\frac{d}{d+1}$};
\draw (2,-0.4) node{$\frac{d}{d+1}$};
\draw[->,line width=1pt] (-0.2,0)--(3.5,0);
\draw[->,line width=1pt] (0,-0.2)--(0,3.5);

\draw[-,line width=1pt] (3,-0.1)--(3,0.1);
\draw[-,line width=1pt] (-0.1,3)--(0.1,3);
\draw[-,line width=1pt] (-0.1,2)--(0.1,2);
\draw[-,line width=1pt] (2,-0.1)--(2,0.1);

\draw[line width=1pt,blue,line width=1pt] (0,0)--(3,0)--(2,2)--(0,3)--(0,0);

\draw[dashed,line width=0.5pt,gray,line width=1pt] (3,0)--(3,3)--(0,3);
\draw (2.2,2.5) node {\small{$(\frac{d}{d+1},\frac{d}{d+1})$}};

\filldraw[blue] (0,0) circle (2pt);
\filldraw[blue] (2,2) circle (2pt);
\filldraw[blue] (3,0) circle (2pt);
\filldraw[blue] (0,3) circle (2pt);
\end{tikzpicture}}
\end{center}
\caption{Region of pairs $(1/p,1/q)$ with $(1/p,1/q,1/r)\in\mathcal{R}_1$, which is the region where we know the H\"older bounds for $\mathcal{T}_1$ are satisfied.}
\end{figure}

\begin{rem}
It is not clear that the corollary above is sharp. Any bound for $\mathcal{T}_1$ of the form $L^p\times L^{q}\rightarrow L^1$, $1/p+1/q>1$, can be lifted to the H\"older bound $L^p\times L^q\rightarrow L^r$, where $1/r=1/p+1/q$ by using Proposition 4.1 in \cite{IPS}), and the bounds for $\mathcal{T}_1$ arriving in $L^1$ are sharp. So Corollary 6 is the best one can get from the lifting bounds into $L^1$ to H\"older bounds, but there could be some $r<1$ for which a H\"older bound $(p,q,r)$ holds but the bound $(p,q,1)$ is false. In the proposition below we include some necessary conditions for the boundedness of $\mathcal{T}_1$ that follow by adapting some examples in \cite{radonplane}.
\end{rem}

The description of the sharp region for which one has $L^p$ improving properties for $\mathcal{T}_1$ would have interesting applications in Sparse domination and weighted bounds for $\mathcal{T}_{lac}$ (see \cite{sparsetriangle},\,\cite{BFOPZ}). In the next proposition we prove some necessary conditions that might be be useful in the future. 

\begin{prop}(Necessary conditions for the boundedness of $\mathcal{T}_1$ in $d\geq 2$) \label{necessarytriangle}
Let $1\leq p,q\leq  \infty$ and $r\in (0,\infty]$. If $$\mathcal{T}_1:L^p\times L^q \rightarrow L^r\text{ is bounded, }$$
then 
\begin{equation}
    \begin{cases}
    \frac{1}{r}\leq \frac{1}{p}+\frac{1}{q}\\
    \frac{1}{p}+\frac{1}{q}\leq \frac{d}{r}\\
    \frac{d}{p}+\frac{1}{q}\leq d-1+\frac{1}{r}\\
    \frac{1}{p}+\frac{d}{q}\leq d-1+\frac{1}{r}
    \end{cases}
\end{equation}

\end{prop}
\begin{rem}
    Observe that the necessary conditions above do not impose restrictions for H\"older bounds for $\mathcal{T}_1$, that is, for any $p,q\in [1,\infty]$ and $1/r=1/p+1/q$ the conditions above are trivially satisfied.
\end{rem}

\begin{proof}[Proof of Proposition \ref{necessarytriangle}]
    
We will give an example for each necessary condition to demonstrate them in order. First take $f_M=\chi_{B(0,2M)}$ where $M>1$. For all $|x|<M$, $\mathcal{T}_1(f_M,f_M)(x)= 1$, so
$$M^{d/r}\lesssim \|\mathcal{T}_1(f_M,f_M)\|_r\lesssim\|f_M\|_p \|f_M\|_q\lesssim M^{d(\frac{1}{p}+\frac{1}{q})}\,\text{ for all }M>1,$$
which implies $1/r\leq 1/p+1/q$.

Next take $f_{\delta}=\chi_{A_{\delta}}$ where $A_{\delta}=\{x\in \R^d\colon 1-\delta\leq|x|\leq 1+\delta\}$. Then for any $|x|<\delta$, 
\begin{equation}
    \begin{split}
        \mathcal{T}_1(f_{\delta}, f_{\delta})(x)=\int_{O(d)}\chi_{A_{\delta}}(x-Ru)\chi_{A_{\delta}}(x-Rv)d\mu(R)=\int_{O(d)}d\mu (R)=1
    \end{split}
\end{equation}
since $1-\delta<1-|x|\leq |x-Ru|\leq 1+|x|< 1+\delta$, for any $R\in O(d)$ since $|u|=1$.
    That gives
    $$\delta^{d/r}\lesssim \|\mathcal{T}_1(f_{\delta},f_{\delta})\|_r\lesssim\|f_{\delta}\|_p \|f_{\delta}\|_q=|A_{\delta}|^{\frac{1}{p}+\frac{1}{q}}\sim{\delta}^{\frac{1}{p}+\frac{1}{q}}\,\text{ for all }\delta\ll 1,$$
    so $1/p+1/q\leq d/r$.

Take $f_{\delta}=\chi_{B(0,\delta)}$ and $g_{\delta}=\chi_{A_{\delta}}$. Suppose $||x|-1|\leq \frac{1}{2}\delta$, then 
\begin{equation}
    \begin{split}
\mathcal{T}_1(f_{\delta},g_{\delta})(x)=&\int_{O(d)} \chi_{B(0,\delta)}(x-Ru)\chi_{A_{\delta}}(x-Rv)d\mu(R)\\
\geq& \int_{O(d)} \chi_{|x-Ru|< \delta} d\mu(R)
    \end{split}
\end{equation}
since for any $|x-Ru|<\delta$, one has $x-Rv\in A_{\delta}$. Therefore 
\begin{equation}
    \begin{split}
\mathcal{T}_1(f_{\delta},g_{\delta})(x)\geq \int_{O(d)} \chi_{B(x,\delta)}(Ru)d\mu(R)=\int_{S^{d-1}} \chi_{B(x,\delta)}(y)d\sigma(y)\gtrsim \delta^{d-1}
    \end{split}
\end{equation}
and from 
$$\delta^{d-1+1/r}\lesssim \|\mathcal{T}_1(f_{\delta},g_{\delta})\|_r\lesssim \|\chi_{B(0,\delta)}\|_p\|g_{\delta}\|_q=\delta^{d/p+1/q}$$
we get $d/p+1/q\leq d-1+1/r$. The last necessary condition follows from the symmetry $\mathcal{T}_1(f,g)=\mathcal{T}_1(g,f)$.
\end{proof}
    
\subsection{Bilinear multiplier boundedness criteria}

For a multiplier $m(\xi, \eta)$ in $\R^{2d}$ we denote by $T_m$ the bilinear operator given by 
$$T_m(f,g)(x)= \int_{\R^{2d}} \hat{f}(\xi)\hat{g}(\eta)m(\xi, \eta)e^{2\pi i x\cdot(\xi +\eta)}d\xi d\eta.$$
These will appear naturally throughout the paper as Fourier transforms of single-scale bilinear averaging operators.

\begin{thm}[\cite{GHS}]\label{ghscriteria}
Let $1\leq q<4$ and set $M_q=\left\lfloor \frac{2d}{4-q} \right\rfloor+1$. Let $m(\xi, \eta)$ be a function in $L^{q}(\R^{2d})\cap \mathcal{C}^{M_q}(\R^{2d})$ satisfying 
\begin{equation}
    \|\partial^{\alpha}m\|_{L^{\infty}}\leq C_0<\infty
\text{ for all multiindices }\alpha \text{ with }|\alpha|\leq M_q.
\end{equation}
Then there is a constant $A$ depending on $d$ and $q$ such that the associated operator $T_m$ with multiplier $m$ satisfies 
\begin{equation}
  \|T_m\|_{L^2\times L^2\rightarrow L^1}\leq A C_0^{1-\frac{q}{4}}\|m\|_{L^q}^{q/4}. 
\end{equation}

\end{thm}

As an immediate corollary of the above theorem, one has the following. 
\begin{cor}\label{corl2l2l1}
Let $m(\xi, \eta )$ be a smooth compactly supported function in $\R^{2d}$ such that 
$$\|\partial^{\alpha}m\|_{L^{\infty}}\leq C_0\,\text{ for all multi-indices } \alpha, \text{ with } |\alpha|\leq \left\lfloor \frac{2d}{3}\right\rfloor+1.$$
Then there exists a constant $A$ depending on $d$ such that the associated operator $T_m$ with multiplier $m$ satisfies
$$\|T_m\|_{L^{2}\times L^{2}\rightarrow L^{1}}\leq A C_0|\text{supp}(m)|^{1/4}.$$

\end{cor}

\section{Decomposition of the operator and bounds for the pieces in the H\"older range}\label{sectiondecayforpieces}

We start by describing the decomposition used in \cite{HHY}. Choose a radial function $\varphi\in \mathcal{S}(\R^d)$ such that $\hat{\varphi}$ is nonnegative, radially decreasing and
\begin{equation}\label{phi def}
\hat{\varphi}(\xi)= \begin{cases}
   1, \text{ if }|\xi|\leq 1\\
   0, \text{ if }|\xi|\geq 2
    \end{cases}.
\end{equation}

Let $\hat{\psi}(\xi)=\hat{\varphi}(\xi)-\hat{\varphi}(2\xi)$, which is supported in $\{1/2<|\xi|<2\}$. Then
\begin{equation}\label{psi def}
\hat{\varphi}(\xi)+\sum_{j=1}^{\infty} \hat{\psi}(2^{-j}\xi)\equiv 1.
\end{equation}

Define for all $i,j\geq 1$.
\begin{equation}
\mathcal{A}_1^{i,j}(f,g)(x)=\int_{\R^{2d}}\hat{f}(\xi)\hat{g}(\eta)\hat{\sigma}(\xi, \eta)\hat{\psi}(2^{-i}\xi)\hat{\psi}(2^{-j}\eta)e^{2\pi i x\cdot (\xi+ \eta)}d\xi d\eta.
\end{equation}

If $i=0$, replace $\hat{\psi}(2^{-i}\xi)$ by $\hat{\varphi}(\xi)$ in the expression above, and similarly if $j=0$. Then one has 

$$\mathcal{A}_1(f,g)(x)=\sum_{i,j\geq 0} \mathcal{A}_1^{i,j}(f,g)(x).$$

\begin{prop}\label{nodecay} Let $p,q\in [1,\infty]$ and $r\in (0,\infty]$, such that $\frac{1}{r}=\frac{1}{p}+\frac{1}{q}$. Then there exists $C>0$ such that for all $i,j\geq 0$
\begin{equation}
    \|\mathcal{A}^{i,j}_1(f,g)\|_{r}\leq C\|f\|_p\|g\|_q.
\end{equation}
\end{prop}

\begin{proof} Denote by $\psi_{i}(x)=2^{id}\psi(2^{i}x)$. Assume $i,j\geq 1$. One can easily see that 
\begin{equation}
    \|\mathcal{A}^{i,j}_1(f,g)\|_{r}=\|\mathcal{A}_1(\psi_i*f,\psi_j*g)\|_{r}.
\end{equation}

Now using Corollary \ref{holderboundsforA1}, 
\begin{equation}
\begin{split}
    \|\mathcal{A}_1(\psi_i*f,\psi_j*g)\|_{r}\leq &C \|\psi_i*f\|_p \|\psi_j*g\|_q\\
    \leq &C \|f\|_p\|g\|_q
\end{split}
\end{equation}
where we used Young's inequality and the fact that $\|\psi_i\|_1=\|\psi\|_1\lesssim 1$. The case where $i=0$ or $j=0$ follows similarly.
\end{proof}

\begin{prop}\label{keyestimate} Let $d\geq 2$. Then for every $i,j\geq 1$
\begin{equation}
    \|\mathcal{A}_1^{i,j}(f,g)\|_1\lesssim 2^{-(i+j)\frac{d-1}{2}}\|f\|_2\|g\|_2.
\end{equation}
\end{prop}

\begin{rem}
The attentive reader might notice that the exact value of the negative power of $2^{i+j}$ that we get in the proposition above is not very important to run the proof of Theorem \ref{boundsspherical} later, so one could also use the estimates for the pieces of $\tilde{\mathcal{M}}$ from \cite{HHY}, or the bilinear multiplier theorem in \cite{GHS} to get the proposition above with a worse negative power. One advantage of the proof below is that it reduces things to spherical averages in $\R^d$ and it can be adapted to take care of more general averaging operators, like the bilinear averages $\mathcal{A}_1^{\mathbf{a}}$ for which we will prove decay bounds in Proposition \ref{decayboundsforleeshuin}. Moreover, having a better decay power in this proposition is interesting on its own, and might have other applications in the study of the sharp $L^p$ improving properties of $\mathcal{A}_1$ for example.
\end{rem}
\begin{proof}
    
Let us use the notation $D_t(f)(x)=f(tx)$. Observe that by using the slicing formula \eqref{slicing formula} followed by polar coordinates in the ball $B^{d}(0,1)$ one gets that
\begin{equation*}
    \begin{split}
        \mathcal{A}_1(f,g)(x)=&\int_{B^{d}(0,1)} f(x-y) \int_{S^{d-1}}g(x-\sqrt{1-|y|^2}z)d\sigma(z) (1-|y|^2)^{\frac{d-2}{2}}dy\\
        =&\int_{0}^{1} \lambda^{d-1}(1-\lambda^2)^{\frac{d-2}{2}}\left\{\int_{S^{d-1}} f(x-\lambda w)d\sigma(w)\right\}\left\{\int_{S^{d-1}}g(x-\sqrt{1-\lambda^2}z)d\sigma(z)\right\} d\lambda\\
    =&\int_{0}^{1}\lambda^{d-1}(1-\lambda^2)^{\frac{d-2}{2}}A_{\lambda}(f)(x)A_{\sqrt{1-\lambda^2}}(g)(x)d\lambda.
    \end{split}
\end{equation*}

Denote $\langle \lambda\rangle=\sqrt{1-\lambda^2} $, and $f^{i}=f*\psi_i$, where $\psi_i$ is like in Proposition \ref{nodecay}. Then 
$$\mathcal{A}_1^{i,j}(f,g)(x)=\mathcal{A}_1(f^{i},g^{j})(x)=\int_{0}^{1}\lambda^{d-1}\langle\lambda\rangle^{d-2}A_{\lambda}(f^{i})(x)A_{\langle\lambda\rangle}(g^{j})(x)d\lambda.$$

Using Minkowski's inequality for integrals followed by H\"older's inequality gives
\begin{equation}
    \begin{split}
\|\mathcal{A}_1(f^{i},g^{j})\|_1\leq&\int_{0}^{1} \lambda^{d-1}\langle\lambda\rangle^{d-2} \|A_{\lambda}(f^{i})\|_2\|A_{\langle\lambda\rangle}(g^{j})\|_2 \,d\lambda.
    \end{split}
\end{equation}

For any $\lambda>0$, 
\begin{equation}\label{linearpiecedecay}
    \begin{split}
        \|A_{\lambda}(f^{i})\|_2=&\|f*\psi_i*d\sigma_{\lambda}\|_2\\
        =&\|\hat{f}(\xi)\hat{\psi}(2^{-i}\xi)\hat{\sigma}(\lambda\xi)\|_2\\
        \lesssim & (\lambda 2^{i})^{-\frac{d-1}{2}}\|f\|_2.
    \end{split}
\end{equation}

Therefore
\begin{equation}
    \begin{split}
\|\mathcal{A}_1(f^{i},g^{j})\|_1\lesssim & 2^{-(i+j)\frac{d-1}{2}}\|f\|_2\|g\|_2\int_{0}^{1} \lambda^{d-1}\langle\lambda\rangle^{d-2} \lambda^{-\frac{d-1}{2}} \langle\lambda\rangle^{-\frac{d-1}{2}}d\lambda\\
=& 2^{-(i+j)\frac{d-1}{2}}\|f\|_2\|g\|_2\int_{0}^{1} \lambda^{\frac{d-1}{2}}\langle\lambda\rangle^{\frac{d-3}{2}}  d\lambda\lesssim 2^{-(i+j)\frac{d-1}{2}}\|f\|_2\|g\|_2.
    \end{split}
\end{equation}
\end{proof}

The following corollary will have a key role in the proof of the bounds for $\mathcal{M}_{lac}$.

\begin{cor}\label{decayforpieces}
    Assume $d\geq 2$. Let $p,q\in (1,\infty)$ and $r\in (0,\infty)$ given by $\frac{1}{r}=\frac{1}{p}+\frac{1}{q}$. Then there exists $\delta(p,q,r,d)$ such that 
    \begin{equation}\label{decayestimate}
        \|\mathcal{A}_{1}^{i,j}(f,g)\|_{r}\lesssim 2^{-(i+j)\delta}\|f\|_p\|g\|_q. 
    \end{equation}

    Moreover, if $d\geq 5$ then for all $q\in (1,\infty)$ and $1/r=1+1/q$,
    \begin{equation}\label{dg5}
    \begin{split}
        \|\mathcal{A}_{1}^{i,j}(f,g)\|_{r}\lesssim  2^{-(i+j)\delta}\|f\|_1\|g\|_q. \\
        \|\mathcal{A}_{1}^{i,j}(f,g)\|_{r}\lesssim  2^{-(i+j)\delta}\|f\|_q\|g\|_1. 
        \end{split}
    \end{equation}
    
\end{cor}

\begin{proof}
    Estimate (\ref{decayestimate}) follows immediately from interpolation of Propositions \ref{nodecay} and \ref{keyestimate}.

    To prove the last part of this corollary, we recall that it was shown in \cite{HHY} that for 
$$\tilde{\mathcal{M}}^{ij}(f,g)(x)=\sup_{t\in[1,2]}|\int_{\R^{2d}} \hat{f}(\xi)\hat{g}(\eta)\hat{\psi}(2^{-i}\xi)\hat{\psi}(2^{-j}\eta)\hat{\sigma}(t\xi,t\eta)e^{2\pi i x\cdot(\xi+\eta)}d\xi d\eta|,$$
one has
$$\|\tilde{\mathcal{M}}^{i,j}(f,g)\|_{L^{2/3}}\lesssim (\max\{2^{i},2^{j}\})^{-\frac{d-4}{2}}\|f\|_{p_1}\|g\|_{p_2}$$
    for all $1\leq p_1,p_2\leq 2 $ with $3/2=1/p_1+1/p_2$.

    In particular, since $\text{max}\{i,j\}\geq \frac{i+j}{2}$
    \begin{equation}
    \begin{split}
    \|\mathcal{A}_1^{i,j}(f,g)\|_{L^{2/3}}&\leq \|\tilde{\mathcal{M}}^{ij}(f,g)\|_{L^{2/3}}\lesssim 2^{-(i+j)\frac{d-4}{4}}\|f\|_{1}\|g\|_{2},\\
    \|\mathcal{A}_1^{i,j}(f,g)\|_{L^{2/3}}&\leq \|\tilde{\mathcal{M}}^{ij}(f,g)\|_{L^{2/3}}\lesssim 2^{-(i+j)\frac{d-4}{4}}\|f\|_{2}\|g\|_{1}.\\
        \end{split}
    \end{equation}
    Then interpolation of these bounds with the bounds in Proposition \ref{nodecay} give the estimates in \eqref{dg5}.
\end{proof}
\begin{rem}
    One could think that estimate \eqref{dg5} could potentially allow one to include the missing boundary pieces in Theorem \ref{boundsspherical} for $d\geq 5$. Unfortunately, that does not seem possible with the methods we exploit in this paper.% one runs into some technical issues trying to do that. 
\end{rem}

\section{Proof of the bounds for the lacunary bilinear spherical maximal operator}\label{prooffirstthm}

\begin{proof}[Proof of Theorem \ref{boundsspherical}]

$\mathcal{M}_{lac}$ satisfies a weak type bound $L^1\times L^{\infty}\rightarrow L^{1,\infty}$ (from \cite{JL} since $\mathcal{M}_{lac}$ is dominated by $\mathcal{M}$).
To see that $\mathcal{M}_{lac}$ can not be bounded from $L^1\times L^{\infty} \rightarrow L^{1}$ one can adapt the example from \cite{JL} to show that for $g\equiv 1$ and $0\leq f\in L^1$
$$\sup_{l\in \Z}\mathcal{A}_{2^{-l}}(f,g)(x)\geq C \sup_{l\in \Z} \frac{1}{2^{-ld}}\int_{B(x,2^{-l})}f(y)dy=:M^{lac}f(x),$$
so the lacunary Hardy-Littlewood maximal function $M^{lac}$ would be bounded in $L^1$ which is false (for any $f\in L^1$  which is not identically zero,  $\frac{1}{|x|^d}\lesssim M^{lac}f(x)$  for $|x|>>1$, so $M^{lac}f\notin L^1$).

 The case $p=\infty$ or $q=\infty$ follows immediately from the known bounds for $\mathcal{M}$ as stated in Theorem \ref{boundsjl}. Thus, we can assume $p,q<\infty$.

The proof will follow from Corollary \ref{decayforpieces} and an adaptation of the ideas in Section 4 of \cite{HHY}.

    Let 
    \begin{equation}
        \mathcal{A}_{2^{-l}}(f,g)(x)=\int_{\R^{2d}} \hat{f}(\xi)\hat{g}(\eta) \hat{\sigma}_{2d-1}(2^{-l}\xi,2^{-l}\eta)e^{2\pi i x\cdot (\xi+\eta)}d\xi d\eta.
    \end{equation}

  One can break down $\mathcal{A}_{2^{-l}}$ into low-low, low-high, high-low and high-high frequencies as follows
  \begin{equation}
      \mathcal{A}_{2^{-l}}(f,g)(x)=-\mathcal{A}_{2^{-l}}^{00}(f,g)(x)+\mathcal{A}_{2^{-l}}^{0\infty}(f,g)(x)+\mathcal{A}_{2^{-l}}^{\infty 0}(f,g)(x)+\sum_{i,j\geq 1}\mathcal{A}_{2^{-l}}^{i,j}(f,g)(x),
  \end{equation}
  where 
  \begin{equation}\label{sphere HL decomp}
      \begin{split}
          \mathcal{A}_{2^{-l}}^{00}(f,g)(x)=&\int_{\R^{2d}}\hat{f}(\xi)\hat{g}(\eta)\hat{\sigma}_{2d-1}(2^{-l}\xi,2^{-l}\eta)\hat{\varphi}(2^{-l}\xi)\hat{\varphi}(2^{-l}\eta)e^{2\pi i x\cdot (\xi+\eta)}d\xi d\eta, \\
  \mathcal{A}_{2^{-l}}^{0\infty}(f,g)(x)=&\int_{\R^{2d}}\hat{f}(\xi)\hat{g}(\eta)\hat{\sigma}_{2d-1}(2^{-l}\xi,2^{-l}\eta)\hat{\varphi}(2^{-l}\xi)e^{2\pi i x\cdot (\xi+\eta)}d\xi d\eta, \\
  \mathcal{A}_{2^{-l}}^{\infty 0}(f,g)(x)=&\int_{\R^{2d}}\hat{f}(\xi)\hat{g}(\eta)\hat{\sigma}_{2d-1}(2^{-l}\xi,2^{-l}\eta)\hat{\varphi}(2^{-l}\eta)e^{2\pi i x\cdot (\xi+\eta)}d\xi d\eta, \\
  \mathcal{A}_{2^{-l}}^{i,j}(f,g)(x)=&\int_{\R^{2d}}\hat{f}(\xi)\hat{g}(\eta)\hat{\sigma}_{2d-1}(2^{-l}\xi,2^{-l}\eta)\hat{\psi}(2^{-i-l}\xi)\hat{\psi}(2^{-j-l}\eta)e^{2\pi i x\cdot (\xi+\eta)}d\xi d\eta. 
      \end{split}
  \end{equation}
  
    Denote $N(i,j;p,q,r)=\|\mathcal{A}_{1}^{i,j}\|_{L^p\times L^q\rightarrow L^r}$, which we know from Corollary \ref{decayforpieces} satisfies 
    $$N(i,j;p,q,r)\lesssim C(p,q,r)2^{-(i+j)\delta}$$
    for some $\delta=\delta(d,p,q,r)>0.$

    Another important fact is the scale-invariance of localized pieces of the average: for all $l\in \Z$, and H\"older exponents $1/r=1/p+1/q$, one has
    \begin{equation}\label{normij}
\|\mathcal{A}_{2^{-l}}^{i,j}\|_{L^p\times L^q\rightarrow L^r}=\|\mathcal{A}_{1}^{i,j}\|_{L^p\times L^q\rightarrow L^r}=:N(i,j;p,q,r).
    \end{equation}
which as observed in \cite{HHY} follows from the fact that 
\begin{equation}
    \mathcal{A}_{2^{-l}}^{i,j}(f,g)(x)=\mathcal{A}_{1}^{i,j}(D_{2^{-l}}f, D_{2^{-l}}g)(2^lx), \text{ for all }i,j\geq 0,
\end{equation}
where $D_{2^{-l}}f(x)=f(2^{-l}x)$.

\begin{prop}\label{lowfrequency} For $d\geq 2$, one has that
\begin{equation}
\begin{split}
 \mathcal{M}^{LL}_{lac}(f,g)(x):=\sup_{l\in\Z}|\mathcal{A}_{2^{-l}}^{00}(f,g)(x)|&\lesssim Mf(x)Mg(x),\\
  \mathcal{M}^{LH}_{lac}(f,g)(x):=  \sup_{l\in\Z} |\mathcal{A}_{2^{-l}}^{0\infty}(f,g)(x)|&\lesssim Mf(x)\sup_{t\in 2^\Z} \left(\int_{S^{2d-1}}|g(x-tz)|d\sigma_{2d-1}(y,z)\right),\\
   \mathcal{M}^{HL}_{lac}(f,g)(x):= \sup_{l\in\Z}| \mathcal{A}_{2^{-l}}^{\infty 0}(f,g)(x)|&\lesssim Mg(x)\sup_{t\in 2^\Z}\left(\int_{S^{2d-1}}|f(x-ty)|d\sigma_{2d-1}(y,z)\right),
\end{split}
\end{equation}
where $M$ stands for the Hardy-Littlewood maximal operator. 
Since 
\begin{equation}
    \sup_{t>0} \left(\int_{S^{2d-1}}|g(x-tz)|d\sigma_{2d-1}(y,z)\right)\lesssim Mg(x),
\end{equation}
one has that
\begin{equation}
[\mathcal{M}^{LL}_{lac}+\mathcal{M}^{LH}_{lac}+\mathcal{M}^{HL}_{lac}](f,g)(x)\lesssim Mf(x)Mg(x).
\end{equation}
\end{prop}

Using H\"older's inequality and the boundedness of $M$ in $L^p$ for all $1<p\leq \infty,$ one gets the following corollary.
\begin{cor}\label{non HH sphere bounds} Let $d\geq2$. Then
    \begin{equation}
\|\mathcal{M}^{LL}_{lac}+\mathcal{M}^{LH}_{lac}+\mathcal{M}^{HL}_{lac}\|_{L^p\times L^q\rightarrow L^r}<\infty
\end{equation}
for all $p,q\in (1,\infty]$, $r\in (0,\infty]$ and $1/r=1/p+1/q$. 
\end{cor}
    
    \begin{proof}[Proof of Proposition \ref{lowfrequency}]
    Let us check the inequality for $\mathcal{M}_{lac}^{LH}$, since the others follow similarly. We follow the arguments in \cite{HHY} most closely.

    Since $(f*\varphi)(x-y)\lesssim Mf(x)$ for all $|y|\leq 1$, one has 
$$|\mathcal{A}_1^{0\infty}(f,g)(x)|\leq \int_{S^{2d-1}}|f*\varphi(x-y)||g(x-z)|d\sigma(y,z)\lesssim Mf(x)\cdot \int_{S^{2d-1}}|g(x-z)|d\sigma(y,z).$$

Therefore
\begin{equation*}
\begin{split}
    \sup_{l\in \Z}|\mathcal{A}_{2^{-l}}^{0\infty}(f,g)(x)|=&\sup_{l\in \Z}|\mathcal{A}_{1}^{0\infty}(D_{2^{-l}}f,D_{2^{-l}}g)(2^lx)|\\
    \lesssim& \sup_{l\in\Z} M(D_{2^{-l}}f)(2^lx)\cdot \int_{S^{2d-1}}|D_{2^{-l}}g(2^lx-z)|d\sigma(y,z)\\
\lesssim&Mf(x)\cdot\sup_{l\in\Z}\int_{S^{2d-1}}|g(x-2^{-l}z)|d\sigma(y,z).
\end{split}
\end{equation*}

Then one just needs to observe that by the slicing formula \eqref{slicing formula} applied with one of the functions being constant, for any $t>0$ one has 
\begin{equation*}
\begin{split}
     \int_{S^{2d-1}} |g(x-tz)|d\sigma(y,z)=&\int_{B^d(0,1)}|g(x-tz)||S^{d-1}|(1-|z|^2)^{\frac{d-2}{2}}dz\\
    \lesssim &\int_{B^d(0,1)}|g(x-tz)|dz\lesssim Mg(x).
\end{split}
\end{equation*}
    \end{proof}
Define $\mathbb{A}_{\mathcal{N}}(f,g)(x)=\sup_{|l|\leq \mathcal{N}}|\mathcal{A}_{2^{-l}}(f,g)(x)|$.
    
Here we observe that since $\|\mathcal{A}_{2^{-l}}\|_{L^p\times L^q\rightarrow L^r}=\|\mathcal{A}_{1}\|_{L^p\times L^q\rightarrow L^r}<\infty$, it is clear that for all $\mathcal{N}\in \N$,
$$A_{\mathcal{N}}(p,q,r):=\|\mathbb{A}_{\mathcal{N}}\|_{L^p\times L^q\rightarrow L^r}<\infty.$$

Our goal is to show that $A_{\mathcal{N}}(p,q,r)$ is bounded by a constant independent of $\mathcal{N}$. Moreover, in view of Corollary \ref{non HH sphere bounds}, it suffices to prove this for the high-high frequency part of the operator. Thus, for every $i,j\geq 1$, define the vector valued operator
\begin{equation}
\mathbb{A}_{\mathcal{N}}^{i,j}:\{f_l\}_{|l|\leq \mathcal{N}}\times \{g_l\}_{|l|\leq \mathcal{N}}\longmapsto \{\mathcal{A}_{2^{-l}}(f_l*\psi_{i+l},g_l*\psi_{j+l})(x)\}_{|l|\leq \mathcal{N}}.
\end{equation}

By the definition of $A_{\mathcal{N}}(p,q,r)$ and the same computation as in \cite{HHY} (page 430), if $p,q>1$ then
\begin{equation}\label{firstbound}
\|\mathbb{A}_{\mathcal{N}}^{i,j}(\{f_l\}_{|l|\leq \mathcal{N}}\times \{g_l\}_{|l|\leq \mathcal{N}})\|_{L^r(\ell^{\infty})}\lesssim A_{\mathcal{N}}(p,q,r) \|\{f_l\}_{|l|\leq \mathcal{N}}\|_{L^{p}(\ell^{\infty})} \| \{g_l\}_{|l|\leq \mathcal{N}}\|_{L^{q}(\ell^{\infty})}.
\end{equation}

The second computation in \cite{HHY}, which exploits the equality (\ref{normij}), gives for $p,q\in (1,\infty)$ and $1/r=1/p+1/q$
\begin{equation}\label{secondbound}
\begin{split}
\|\mathbb{A}_{\mathcal{N}}^{i,j}(\{f_l\}_{|l|\leq \mathcal{N}}\times \{g_l\}_{|l|\leq \mathcal{N}})\|_{L^r(\ell^{r})}\lesssim& N(i,j;p,q,r) \|\{f_l\}_{|l|\leq \mathcal{N}}\|_{L^{p}(\ell^{p})} \| \{g_l\}_{|l|\leq \mathcal{N}}\|_{L^{q}(\ell^{q})}\\
\lesssim& N(i,j;p,q,r) \|\{f_l\}_{|l|\leq \mathcal{N}}\|_{L^{p}(\ell^{1})} \| \{g_l\}_{|l|\leq \mathcal{N}}\|_{L^{q}(\ell^{1})}.
\end{split}
\end{equation}

Interpolation of the bounds in (\ref{firstbound}) and (\ref{secondbound}) will then imply that for any 
$$p,q\in (1,\infty), \text{ and }1/r=1/p+1/q  $$
one has
\begin{equation}\label{thirdbound}
\begin{split}
\|\mathbb{A}_{\mathcal{N}}^{i,j}(\{f_l\}_{|l|\leq \mathcal{N}}\times &\{g_l\}_{|l|\leq \mathcal{N}})\|_{L^r(\ell^{2r})}\\
\lesssim &A_{\mathcal{N}}(p,q,r)^{1/2} N(i,j;p,q,r)^{1/2} \|\{f_l\}_{|l|\leq \mathcal{N}}\|_{L^{p}(\ell^{2})} \| \{g_l\}_{|l|\leq \mathcal{N}}\|_{L^{q}(\ell^{2})}.
\end{split}
\end{equation}

 Using bound (\ref{thirdbound}) and the Littlewood-Paley Theorem (Theorem 6.1.2 in \cite{classicalgrafakos}), for any $p,q>1$
\begin{equation}
    \begin{split}
        \left\|\sup_{|l|\leq \mathcal{N}} |\mathcal{A}_{2^{-l}}^{i,j}(f,g)|\right\|_{L^{r}(\R^d)}=& \left\|\sup_{|l|\leq \mathcal{N}} |\mathcal{A}_{2^{-l}}(f*\psi_{i+l},g*\psi_{j+l})|\right\|_{L^r(\R^d)}\\
=&\left\|\mathbb{A}_{\mathcal{N}}^{i,j} (\{f*\psi_{i+l}\}_{|l|\leq \mathcal{N}}\times \{g*\psi_{j+l}\}_{|l|\leq \mathcal{N}})\right\|_{L^r(\ell^{\infty})}\\
\le&\left\|\mathbb{A}_{\mathcal{N}}^{i,j} (\{f*\psi_{i+l}\}_{|l|\leq \mathcal{N}}\times \{g*\psi_{j+l}\}_{|l|\leq \mathcal{N}})\right\|_{L^r(\ell^{2r})}\\
\lesssim& A_{\mathcal{N}}(p,q,r)^{1/2}N(i,j;p,q,r)^{1/2}\\
&\,\,\left\|\{f*\psi_{i+l}\}_{|l|\leq \mathcal{N}}\right\|_{L^p(\ell^2)}\left\|\{g*\psi_{j+l}\}_{|l|\leq \mathcal{N}}\right\|_{L^q(\ell^2)}\\
 \lesssim&A_{\mathcal{N}}(p,q,r)^{1/2}N(i,j;p,q,r)^{1/2}\|f\|_p\|g\|_q.
    \end{split}
\end{equation}

\begin{rem} The argument fails for $p=1$ since with Littlewood-Paley theory we can only control $\|\{f*\psi_{i+l}\}_{|l|\leq \mathcal{N}}\|_{L^{1,\infty}(\ell^2)}\lesssim \|f\|_1$ but not $\|\{f*\psi_{i+l}\}_{|l|\leq \mathcal{N}}\|_{L^{1}(\ell^2)}$. Even if one tries to use this weak type control of the Littlewood-Paley pieces, one runs into the issue there are no bounds of the form
$\|\mathcal{A}_{1}^{i,j}(f,g)\|_{L^{r,\infty}}\leq N(i,j;1,q,r)\|f\|_{L^{1,\infty}} \|g\|_q$.
\end{rem}

To finish the proof for the interior points $p,q\in (1,\infty)$, we split into two cases for $r$. When $r<1$, we use that $\|\cdot\|_{r}^{r}$ satisfies the triangle inequality, and for $r\geq 1$, we use that $\|\cdot\|_r$ is a norm.

\textbf{Case $r<1$}:
\begin{equation}
    \begin{split}
        A_{\mathcal{N}}(p,q,r)^r=&\|\sup_{|l|\leq \mathcal{N}}|\mathcal{A}_{2^{-l}}|\|^r_{L^p\times L^q\rightarrow L^r}\\
\leq&\|\mathcal{M}_{lac}^{LL}+\mathcal{M}_{lac}^{LH}+\mathcal{M}_{lac}^{HL}\|_{L^p\times L^q\rightarrow L^r}^{r}+\|\sum_{i,j\geq 1} \sup_{|l|\leq \mathcal{N}}|\mathcal{A}_{2^{-l}}^{i,j}|\|_{L^p\times L^q\rightarrow L^r}^r\\
\lesssim &1+\sum_{i,j\geq 1} \|\sup_{|l|\leq \mathcal{N}}|\mathcal{A}_{2^{-l}}^{i,j}|\|_{L^p\times L^q\rightarrow L^r}^r \\
\lesssim & 1+A_{\mathcal{N}}(p,q,r)^{r/2}\sum_{i,j\geq 1}N(i,j;p,q,r)^{r/2}.
\end{split}
\end{equation}

Hence, we deduce
\begin{equation}
 A_{\mathcal{N}}(p,q,r)\lesssim 1 +\left\{\sum_{i,j\geq 1} N(i,j;p,q,r)^{r/2}\right\}^{2/r}.
\end{equation}

By Corollary \ref{decayforpieces} one has 
$$\sum_{i,j\geq 1} N(i,j;p,q,r)^{r/2}\lesssim \sum_{i,j\geq 1} 2^{-(i+j)\delta r/2}\lesssim 1.$$

The case $r\geq 1$ is similar but we use the triangle inequality for the norm $\|\cdot\|_r$ to take care of the sum in $i,j$.

Therefore for any $p,q\in (1,\infty)$ and $1/r=1/p+1/q$
$$\sup_{\mathcal{N\in \N}} A_{\mathcal{N}}(p,q,r)\leq C$$
and this finishes the proof.

\end{proof}

\section{Proof of the bounds for the lacunary triangle averaging maximal operator}\label{proofsecondthm}

In analogy with what we had for the bilinear spherical averaging operator define
\begin{equation}
      \begin{split}
          \mathcal{T}_{l}^{00}(f,g)(x)=&\int_{\R^{2d}}\hat{f}(\xi)\hat{g}(\eta)\hat{\mu}(2^{-l}\xi,2^{-l}\eta)\hat{\varphi}(2^{-l}\xi)\hat{\varphi}(2^{-l}\eta)e^{2\pi i x\cdot (\xi+\eta)}d\xi d\eta, \\
  \mathcal{T}_{l}^{0\infty}(f,g)(x)=&\int_{\R^{2d}}\hat{f}(\xi)\hat{g}(\eta)\hat{\mu}(2^{-l}\xi,2^{-l}\eta)\hat{\varphi}(2^{-l}\xi)e^{2\pi i x\cdot (\xi+\eta)}d\xi d\eta, \\
  \mathcal{T}_{l}^{\infty 0}(f,g)(x)=&\int_{\R^{2d}}\hat{f}(\xi)\hat{g}(\eta)\hat{\mu}(2^{-l}\xi,2^{-l}\eta)\hat{\varphi}(2^{-l}\eta)e^{2\pi i x\cdot (\xi+\eta)}d\xi d\eta, \\
  \mathcal{T}_{l}^{i,j}(f,g)(x)=&\int_{\R^{2d}}\hat{f}(\xi)\hat{g}(\eta)\hat{\mu}(2^{-l}\xi,2^{-l}\eta)\hat{\psi}(2^{-i-l}\xi)\hat{\psi}(2^{-j-l}\eta)e^{2\pi i x\cdot (\xi+\eta)}d\xi d\eta,
      \end{split}
  \end{equation}
and let $\mathcal{T}_{lac}^{LL},\,\mathcal{T}_{lac}^{LH},\,\mathcal{T}_{lac}^{HL}$ be

\begin{equation}
\begin{split}
 \mathcal{T}^{LL}_{lac}(f,g)(x):=&\sup_{l\in\Z}|\mathcal{T}_{l}^{00}(f,g)(x)|,\\
  \mathcal{T}^{LH}_{lac}(f,g)(x):=&\sup_{l\in\Z} |\mathcal{T}_{l}^{0\infty}(f,g)(x)|,\\
   \mathcal{T}^{HL}_{lac}(f,g)(x):=&\sup_{l\in\Z} |\mathcal{T}_{l}^{\infty 0}(f,g)(x)|.
\end{split}
\end{equation}

Recall that we are assuming $d\geq 2$ and $\mathcal{R}_1$ is the region inside $[0,1]\times [0,1]$ given by the closure of the points $(0,0),\,(0,1),\,(1,0)$ and $(\frac{d}{d+1},\frac{d}{d+1})$.

With similar arguments as the lacunary bilinear maximal function the proof will reduce to some key ingredients. First we will need to guarantee that the low-low, low-high, and high-low parts of the operator are bounded. Secondly, it is immediate that the analogue of the estimate in Proposition \ref{nodecay} holds for the pieces of the triangle averaging operator $\mathcal{T}_1$ for exponents $(1/p,1/q)\in \mathcal{R}_1$, where we know $\mathcal{T}_1$ is bounded by Corollary \ref{holderboundsforT1}. We state this as a proposition.

\begin{prop}\label{nodecaytriangle} Let $p,q\in [1,\infty]$ and $r\in (0,\infty]$, such that $\frac{1}{r}=\frac{1}{p}+\frac{1}{q}$. If $(1/p,1/q)\in \mathcal{R}_1$, then there exists $C>0$ such that for all $i,j\geq 0$
\begin{equation}
    \|\mathcal{T}^{i,j}_1(f,g)\|_{r}\leq C\|f\|_p\|g\|_q.
\end{equation}
\end{prop}

Thirdly, if we prove a decay bound for $\mathcal{T}^{i,j}$ in the triple $(p,q,r)=(2,2,1)$, then by interpolation with the bound in Proposition \ref{nodecaytriangle} we will have decay in $i,j$ throughout the boundedness region $\mathcal{R}_1$. Hence, we still need to control the low-low, low-high, high-low terms and establish the decay bound at the exponent $ (p,q,r)=(2,2,1)$.

\begin{prop}[Controlling LL, LH, HL parts]\label{lowtriangle}
    For any $d\geq 2$, one has 
    \begin{equation}
[\mathcal{T}_{lac}^{LL}+\mathcal{T}_{lac}^{LH}+\mathcal{T}_{lac}^{HL}](f,g)(x)\lesssim Mf(x)\mathcal{S}_{lac}(g)(x)
    \end{equation}
    where $\mathcal{S}_{lac}$ stands for the lacunary spherical maximal function.
\end{prop}

\begin{proof}
For any $l\in \Z$, 
\begin{equation}\label{trianglelhatl}
    \begin{split}
         \mathcal{T}_l^{0\infty}(f,g)(x)=&\int_{\R^{2d}} \hat{f}(\xi)\hat{g}(\eta) \hat{\varphi}(2^{-l}\xi)\hat{\mu}(2^{-l}\xi,2^{-l}\eta)e^{2\pi ix\cdot(\xi+\eta)}d\xi d\eta\\
         =&\mathcal{T}_{0}^{0\infty}(D_{2^{-l}}f,D_{2^{-l}}g)(2^lx)
    \end{split}
\end{equation}
by a simple change of variables. So let us look at the case $l=0$.
\begin{equation*}
\begin{split}
    |\mathcal{T}_{0}^{0\infty}(f,g)(x)|\leq& \int_{\mathcal{I}} |f*\varphi(x-y)||g(x-z)|d\mu(y,z)\\
    \lesssim &Mf(x)\int_{\mathcal{I}} |g(x-z)|d\mu(y,z).
\end{split}
\end{equation*}

   Recall that (see \cite{IPS}) fixing $u,v\in S^{d-1}$ with $|u-v|=1$, one has
   $$\int_{\mathcal{I}}h(x-y)g(x-z)d\mu(y,z)=\int_{O(d)}h(x-Ru)g(x-Rv)dR.$$

In particular, for $h\equiv 1$,
\begin{equation}
    \begin{split}
        \int_{\mathcal{I}} |g(x-z)|d\mu(y,z)=&\int_{O(d)}|g(x-Rv)|dR\\
        = &A_1(|g|)(x)
    \end{split}
\end{equation}
where $A_t$ is the spherical average $A_tg(x)=\int_{S^{d-1}}g(x-ty)d\sigma(y)$. Hence,
$$|\mathcal{T}_0^{0\infty}(f,g)(x)|\lesssim Mf(x)\cdot A_1(|g|)(x).$$

Going back to equality (\ref{trianglelhatl}) one gets that 
\begin{equation}
    \begin{split}
      \sup_{l\in\Z}  |\mathcal{T}_l^{0\infty}(f,g)(x)|
         \lesssim& \sup_{l\in\Z}M(D_{2^{-l}}f)(2^lx)\cdot A_1(|D_{2^{-l}}g|)(2^lx)\\
         \lesssim & Mf(x)\cdot \sup_{l\in\Z}A_{2^{-l}}(|g|)(x)=Mf(x)\cdot \mathcal{S}_{lac}g(x).
    \end{split}
\end{equation}
\end{proof}

Recall that $\mathcal{S}_{lac}$ is bounded in $L^p$ for any $1<p\leq \infty$. Combining that with the boundedness properties of the Hardy-Littlewood maximal function $M$, the following corollary follows immediately from Proposition \ref{lowtriangle}. 

\begin{cor} For any $p,q\in(1,\infty]$, $r\in (0,\infty]$ and $1/r=1/p+1/q$, one has
$$\|\mathcal{T}_{lac}^{LL}+\mathcal{T}_{lac}^{LH}+\mathcal{T}_{lac}^{HL}\|_{L^p\times L^q\rightarrow L^r}<\infty.$$
\end{cor}

Define for all $i,j\geq 1$.
\begin{equation}
\mathcal{T}_1^{i,j}(f,g)(x)=\int_{\R^{2d}}\hat{f}(\xi)\hat{g}(\eta)\hat{\mu}(\xi, \eta)\hat{\psi}(2^{-i}\xi)\hat{\psi}(2^{-j}\eta)e^{2\pi i x\cdot (\xi+ \eta)}d\xi d\eta.
\end{equation}

We have the following result which follows directly from the methods in \cite{triangleaveraging}. We defer the details to the appendix.

\begin{prop} \label{PS ij estimate}  Let $d\geq 7$.  Then there exists $\delta=\delta(d)=\dfrac{3d-20}{32}>0$ such that for all $i,j\geq 1$
    \begin{equation}
      \|\mathcal{T}_{1}^{i,j}(f,g)\|_1\leq C 2^{-(i+j)\delta} \|f\|_2\|g\|_2.
      \end{equation}
\end{prop}

This is summable, so repeating the argument in section \ref{prooffirstthm} gives the claim.

\section{Lee-Shuin class of averaging operators associated to degenerate surfaces}\label{leeshuinsection}

Let $d\geq 2$ and $\mathbf{a}=(\mathbf{a}_1,\mathbf{a}_2)\in [1,\infty)^2$. Denote by $\mathcal{A}_1^{\mathbf{a}}$ the average corresponding to the compact surface in Lee and Shuin's paper \cite{leeshuin}, which we denote by $\mathcal{S}^{\mathbf{a}}$:
\begin{equation}
\mathcal{S}^{\mathbf{a}}=\{(y,z)\in \R^d\times \R^d\colon |y|^{\mathbf{a}_1}+|z|^{\mathbf{a}_2}=1\}
\end{equation}
and 
\begin{equation}
    \mathcal{A}_{t}^{\mathbf{a}}(f,g)(x)=\int_{\mathcal{S}^{\mathbf{a}}} f(x-ty)g(x-tz)d\mu^{\mathbf{a}}(y,z)
\end{equation}
where $\mu^{\mathbf{a}}$ is the surface measure in $\mathcal{S}^{\mathbf{a}}$.

Our goal will be to state the bound for the associated lacunary maximal function in terms of the H\"older boundedness region of $\mathcal{A}_1^{\mathbf{a}}$, namely,
$$\mathcal{R}^{\mathbf{a}}=\left\{\left(\frac{1}{p},\frac{1}{q},\frac{1}{r}\right)\colon (p,q,r) \in[1,\infty]^2\times (0,\infty]\colon 1/r=1/p+1/q\text{ and }  \|\mathcal{A}^{\mathbf{a}}_1\|_{L^p\times L^{q}\rightarrow L^r}<\infty \right\}.$$

Fix $\mathbf{a}_1,\mathbf{a}_2\geq1$ and $\mathbf{a}=(\mathbf{a}_1,\mathbf{a}_2)$. Let $\Phi^{\mathbf{a}}(y,z)=|y|^{\mathbf{a}_1}+|z|^{\mathbf{a}_2}-1$. One has $|\nabla \Phi^{\mathbf{a}}(y,z)|^2=\mathbf{a}_1^2|y|^{2(\mathbf{a}_1-1)}+\mathbf{a}_2^2|z|^{2(\mathbf{a}_2-1)}\neq 0$ for all $(y,z)\in \mathcal{S}^{\mathbf{a}}=(\Phi^{\mathbf{a}})^{-1}(\{0\})$ so there exists $c^{\mathbf{a}}>0$ and $C^{\mathbf{a}}>0$ such that $c^{\mathbf{a}}\leq |\nabla \Phi^{\mathbf{a}}(y,z)|\leq C^{\mathbf{a}}$ for $(y,z)\in \mathcal{S}^{\mathbf{a}}$.% (\textcolor{red}{or at least outside a set of zero $(2d-1)$-Hausdorff measure}).

The $L^p\times L^q \rightarrow L^r$ bounds for
\begin{equation}
    \mathcal{M}^{\mathbf{a}}_{lac}(f,g)(x):=\sup_{t\in 2^\Z} |\mathcal{A}_t^{\mathbf{a}}(f,g)(x)|
\end{equation}
are the same as those of the operator 
\begin{equation}
    \mathcal{N}^{\mathbf{a}}_{lac}(f,g)(x):=\sup_{t\in 2^\Z} |\mathcal{B}_t^{\mathbf{a}}(f,g)(x)|,
\end{equation}
where 
\begin{equation}
    \mathcal{B}_{t}^{\mathbf{a}}(f,g)(x)=\int_{\mathcal{S}^{\mathbf{a}}} f(x-ty)g(x-tz)\dfrac{d\mu^{\mathbf{a}}(y,z)}{|\nabla\Phi^{\mathbf{a}}(y,z)|}.
\end{equation}

For any $i,j\geq 0$, define 
\begin{equation}
    (\mathcal{B}_{1}^{\mathbf{a}})^{i,j}(f,g)=\mathcal{B}_1^{\mathbf{a}} (f^{i}, g^{j})
\end{equation}
where for $i\geq 1$, $\hat{f^{i}}(\xi)=\hat{f}(\xi)\hat{\psi}(2^{-i}\xi)$ and for $i=0$, $\hat{f^{0}}(\xi)=\hat{f}(\xi)\hat{\varphi}(\xi)$. In other words, 
$f^{i}=f*\psi_i$ where $\psi_i(x)=2^{id}\psi(2^{i}x)$ for $i\geq 1$, and $f^{0}=f*\varphi$.
\begin{prop}\label{decayboundsforleeshuin}
      Assume $d\geq 2$. Let $(\frac{1}{p},\frac{1}{q},\frac{1}{r})\in \text{int}(\mathcal{R}^{\mathbf{a}})$. Then there exists $\delta(p,q,r,d)>0$ such that for all $i,j\geq 1$
    \begin{equation}\label{decayestimate2}
        \|(\mathcal{B}_1^{\mathbf{a}})^{i,j}(f,g)\|_{r}\leq C 2^{-(i+j)\delta}\|f\|_p\|g\|_q. 
    \end{equation}
\end{prop}

\begin{proof}

    By interpolation with the bound 
    $$\|(\mathcal{B}_{1}^{\mathbf{a}})^{i,j}(f,g)\|_{r}\leq C\|f\|_p\|g\|_q,\,\forall\, \left(\frac{1}{p},\frac{1}{q},\frac{1}{r}\right)\in \mathcal{R}^{\mathbf{a}}$$
    it is enough to prove that there exists $\delta>0$ such that
    \begin{equation}
        \|(\mathcal{B}_{1}^{\mathbf{a}})^{i,j}(f,g)\|_{1}\leq C 2^{-(i+j)\delta}\|f\|_2\|g\|_2.
    \end{equation}
     
We do it by generalizing the proof given in Proposition \ref{keyestimate}, which corresponds to the case $\mathbf{a}_1=\mathbf{a}_2=2$.

For $y\in \R^d$ fixed, let
\begin{equation*}
    \begin{split}
\Phi^{\mathbf{a}}_y(z)=&|y|^{\mathbf{a}_1}+|z|^{\mathbf{a}_2}-1 \\
\Omega_y=&(\Phi_y^{\mathbf{a}})^{-1}(0).
    \end{split}
\end{equation*}

For all $z\in \Omega_y$, $|\nabla\Phi_y^{\mathbf{a}} (z)|=\mathbf{a}_2|z|^{\mathbf{a}_2-1}=\mathbf{a}_2(1-|y|^{\mathbf{a}_1})^{\frac{\mathbf{a}_2-1}{\mathbf{a}_2}}=:\mathbf{a}_2\omega_{\mathbf{a}}(y)^{\mathbf{a}_2-1}$.

\begin{equation}
    \begin{split}\label{slicingleeshuin}
\mathcal{B}_1^{\mathbf{a}}(f,g)(x)=&\int_{\mathcal{S}^{\mathbf{a}}} f(x-y)g(x-z)\dfrac{d\mu_{\mathbf{a}}(y,z)}{|\nabla \Phi^{\mathbf{a}}(y,z)|}\\
        =& \int_{\R^{2d}} f(x-y)g(x-z)\delta(\Phi^{\mathbf{a}})dydz\\
        =&\int_{|y|<1 } f(x-y)\left(\int g(x-z)\delta(\Phi_y^{\mathbf{a}})dz\right)dy\\
        =&\int_{|y|< 1} f(x-y)\left(\int_{\Omega_y} g(x-z)\dfrac{d\sigma_y(z)}{|\nabla\Phi_y^{\mathbf{a}}(z)|}\right) dy\\
        =&\int_{|y|< 1} f(x-y)\left(\frac{1}{\omega_{\mathbf{a}}(y)^{d-1}}\int_{|z|=\omega_{\mathbf{a}}(y)} g(x-z)d\sigma_y(z)\right) \frac{\omega_{\mathbf{a}}(y)^{d-\mathbf{a}_2}}{\mathbf{a}_2}dy.
    \end{split}
\end{equation}

Denote 
$$Af(x,t):=A_tf(x)=\int_{S^{d-1}}f(x-ty)d\sigma(y).$$

Then 
\begin{equation}
    \begin{split}
    \mathcal{B}_{1}^{\mathbf{a}}(f,g)(x)= \frac{1}{\mathbf{a}_2}\int_{|y|< 1} f(x-y) Ag(x,\omega_{\mathbf{a}}(y)){\omega_{\mathbf{a}}(y)^{d-\mathbf{a}_2}}dy.
    \end{split}
\end{equation}

If we use polar coordinates in the unit ball this says that 
\begin{equation}
    \begin{split}
    \mathcal{B}_{1}^{\mathbf{a}}(f,g)(x)=& \frac{1}{\mathbf{a}_2}\int_{0}^{1} t^{d-1} \int_{S^{d-1}}f(x-t\omega)d \sigma(\omega) Ag(x,(1-t^{\mathbf{a}_1})^{1/\mathbf{a}_2}){(1-t^{\mathbf{a}_1})^{\frac{d-\mathbf{a}_2}{\mathbf{a}_2}} }dt\\
    =&\frac{1}{\mathbf{a}_2}\int_{0}^{1}  Af(x,t) Ag(x,(1-t^{\mathbf{a}_1})^{1/\mathbf{a}_2})t^{d-1}{(1-t^{\mathbf{a}_1})^{\frac{d-\mathbf{a}_2}{\mathbf{a}_2}} }dt.
    \end{split}
\end{equation}

Applying this equality for $f^{i}$ and $g^{j}$ and using Minkowski's inequality, one gets
\begin{equation}
    \begin{split}
        \|\mathcal{B}_{1}^{\mathbf{a}}(f^{i},g^{j})\|_{1}\leq & \frac{1}{\mathbf{a}_2}  \int_{0}^{1}  \|A(f^{i})(x,t)\|_{L^{2}_x} \|A(g^{j})(x,(1-t^{\mathbf{a}_1})^{1/\mathbf{a}_2})\|_{L^{2}_x}t^{d-1}{(1-t^{\mathbf{a}_1})^{\frac{d-\mathbf{a}_2}{\mathbf{a}_2}}}dt  \\
        \lesssim&  2^{-(i+j)\frac{d-1}{2}}\|f\|_2\|g\|_2 \int_{0}^{1} t^{d-1-\frac{d-1}{2}}\left\{(1-t^{\mathbf{a}_1})^{1/\mathbf{a}_2}\right\}^{d-\mathbf{a}_2-\frac{d-1}{2}} dt\\
        \leq &2^{-(i+j)\frac{d-1}{2}}\|f\|_2\|g\|_2 \int_{0}^{1} \left\{1-t^{\mathbf{a}_1}\right\}^{\frac{d-2\mathbf{a}_2+1}{2\mathbf{a}_2}} dt.\\
        \end{split}
\end{equation}

Observe that in the estimate above we used estimate (\ref{linearpiecedecay}) again, which gives
\begin{equation}
        \|A(f^{i})(\cdot,t)\|_2
        \lesssim  (t 2^{i})^{-\frac{d-1}{2}}\|f\|_2.
\end{equation}

One can check that for $\mathbf{a}_1\geq 1$, $\int_{0}^{1}(1-t^{\mathbf{a}_1})^{p}dt<\infty$ if $p>-1$. 

Since $\frac{d-2\mathbf{a}_2+1}{2\mathbf{a}_2}>-1$, we get
$$ \|\mathcal{B}_{1}^{\mathbf{a}}(f^{i},g^{j})\|_{1}\lesssim 2^{-(i+j)\frac{d-1}{2}}\|f\|_2\|g\|_2.$$

\end{proof}
To prove Theorem \ref{lacunaryleeshuin} the argument is similar as before, for $i,j\geq 1$ one considers pieces like
\begin{equation}
\begin{split}
   (\mathcal{B}_{2^{-l}}^{\mathbf{a}})^{i,j}(f,g)(x):=&\mathcal{B}_{2^{-l}}^{\mathbf{a}}(f*\psi_{i+l},g*\psi_{j+l})(x)\\
   =&\mathcal{B}_1^{\mathbf{a}}(D_{2^{-l}}(f*\psi_{i+l}),D_{2^{-l}}(g*\psi_{j+l}))(2^l x).
   \end{split}
\end{equation}

The previous proposition allow us to control the high-high pieces of the maximal operator, but we also need to make sure we can control the low-low, low-high, and high-low parts as well.

Let us illustrate with the high-low part
\begin{equation}
    \begin{split}
    \mathcal{N}_{lac}^{HL}(f,g)(x)\mathcal:=\sup_{l\in \Z}|\mathcal{B}^{\mathbf{a}}_{2^{-l}}(f,g*\varphi_l)|(x).
    \end{split}
\end{equation}

 For any $l\in \Z$, using that $D_{2^{-l}}(g*\varphi_l)=(D_{2^{-l}}g)*\varphi$
 \begin{equation}
     \begin{split}
         |\mathcal{B}_{2^{-l}}^{\mathbf{a}}(f, g*\varphi_l)(x)|=&|\mathcal{B}_{1}^{\mathbf{a}}(D_{2^{-l}}f,D_{2^{-l}}(g*\varphi_l))(2^lx)|\\
        =&\int_{\mathcal{S}^{\mathbf{a}}}f(x-2^{-l}y)|(D_{2^{-l}}g*\varphi)(2^lx-z)| \dfrac{d\mu^{\mathbf{a}}(y,z)}{|\nabla \Phi^{\mathbf{a}}(y,z)|}\\
        \lesssim & M(D_{2^{-l}}g)(2^lx) \int_{\mathcal{S}^{\mathbf{a}}} f(x-2^{-l}y)\dfrac{d\mu^{\mathbf{a}}(y,z)}{|\nabla \Phi^{\mathbf{a}}(y,z)|}\\
        \lesssim & M(g)(x) \mathcal{B}_{2^{-l}}(f,1)(x).
     \end{split}
 \end{equation}
 
With the slicing in equation (\ref{slicingleeshuin}), one can see that since $\mathbf{a_2}\leq d$ then $\omega_{\mathbf{a}}(y)^{d-\mathbf{a_2}}\leq 1$ and
\begin{equation}\label{HLatscalel}
    \mathcal{B}_{2^{-l}}(f,1)(x) \lesssim \int_{|y|\leq 1} f(x-2^{-l}y)dy\lesssim  Mf(x).
\end{equation}

Therefore, 
$$\mathcal{N}^{HL}_{lac}(f,g)(x)\lesssim Mf(x)Mg(x).$$

Similarly since we are assuming $\mathbf{a_1}\leq d$ then $$\mathcal{N}^{LH}_{lac}(f,g)(x)\lesssim Mf(x)Mg(x).$$

\begin{rem}
    The assumption $\mathbf{a_2}\leq d$ in Theorem \ref{lacunaryleeshuin}  was only used in (\ref{HLatscalel}) to control the high-low part of $\mathcal{N}_{lac}^{\mathbf{a}}$. Similarly $\mathbf{a_1}\leq d$ was only used to control the low-high part. We believe those assumptions can be weakened, but we will not pursue this in this paper.
\end{rem}

\section*{Appendix: Proof of Proposition \ref{PS ij estimate}}
For the reader's convenience, we present the details as to how the argument from \cite{triangleaveraging} implies Proposition \ref{PS ij estimate}.
\begin{proof} We are going to adapt some of the arguments in \cite{triangleaveraging} but using a slightly different decomposition to better match our notation in this paper. We separately localize the variables $|\xi|\sim 2^{i}$ and $|\eta|\sim 2^{j}$. The main tools will be Theorem \ref{ghscriteria} and Corollary \ref{corl2l2l1}. 

We start by further decomposing the operator according to the angle $\theta=\theta(\xi,\eta)$ between $\xi$ and $\eta$, which is given by $\cos(\theta)=\frac{\xi\cdot\eta}{|\xi||\eta|}$ for every $\xi\neq 0$ and $\eta \neq 0$.
\begin{equation}
\mathcal{T}_1^{i,j}=\sum_{k\geq 0} \mathcal{T}_1^{i,j,k}
\end{equation}
where the bilinear multiplier associated to $\mathcal{T}_1^{i,j,k}$ is localized at $|\sin(\theta)|\sim 2^{-k}$. More precisely, 
$$\mathcal{T}_{1}^{i,j,k}=\mathcal{T}_{m_{i,j,k}}$$ 
where for $i,j\geq 1$
\begin{equation}
    m_{i,j,k}(\xi, \eta)=\hat{\mu}(\xi,\eta)\hat{\psi}(2^{-i}\xi)\hat{\psi}(2^{-j}\eta)\rho_k(\xi, \eta).
\end{equation}

Here, $\psi$ is the same as defined in \eqref{phi def}, \eqref{psi def} and $\rho_k$ are smooth functions with $\sum_{k\geq 0}\rho_k(\xi,\eta)\equiv1$ except at the origin and
\begin{equation}
\begin{split}
\text{supp}(\rho_k)&\subseteq\{(\xi, \eta)\colon 2^{-k-1}\leq |\sin(\theta)|\leq 2^{-k+1}\},\,\text{if }k\geq1;\\
    \text{supp}(\rho_0)&\subseteq\{(\xi, \eta)\colon |\sin(\theta)|\geq 1/2\}.
\end{split}
\end{equation}

We can choose such a partition in such a way that when one defines 
\begin{equation}
    \rho^{l}(\xi,\eta)=\sum_{k=l}^{\infty} \rho_k(\xi,\eta),
\end{equation}
then $\text{supp}(\rho^l)$ is contained in $\{(\xi, \eta)\colon |\sin(\theta)|\leq 2^{-l+1}\}$ and $\rho^l\equiv 1$  in $\{(\xi, \eta)\colon |\sin(\theta)|\leq 2^{-l}\}$.
In particular, that implies that the support of any derivative of $\rho^l$ is contained in $\{(\xi,\eta)\colon 2^{-l}\leq|\sin(\theta)|\leq 2^{-l+1}\}$. The interested reader may find more details on how to construct such a partition of unity $\{\rho_k\}$ in \cite{triangleaveraging}.
 
One can observe that for $i,j,k\geq 1$
\begin{equation*}
    \text{supp}(m_{i,j,k})
    \subseteq
        \{(\xi,\eta)\in \R^{2d}\colon 2^{i-1}\leq |\xi|\leq 2^{i+1},2^{j-1}\leq |\eta|\leq 2^{j+1},2^{-k-1}\leq |\sin(\theta)|\leq 2^{-k+1}\}.
\end{equation*}

This allow us to estimate $|\text{supp}(m_{i,j,k})|$.
\begin{equation}\label{support}
\begin{split}
     |\text{supp}(m_{i,j,k})|\leq & |\{(\xi,\eta)\in \R^{2d}\colon |\xi|\leq 2^{i+1}, |\eta|\leq 2^{j+1},|\sin(\theta)|\leq 2^{-k+1}\}|\\
    \lesssim &2^{(i+j)d}2^{-k(d-1)}.
\end{split}
\end{equation}

The final estimate above follows in a straightforward manner by noticing that for a fixed $\xi$ with $|\xi|\le 2^{i+1}$, the admissible points $\eta$ lie in a sector of a ball of radius $2^{j+1}$ and angle at most $C2^{-k}$; for the details (with a slightly modified definition of the multiplier), see \cite{triangleaveraging}.

For $k=0$ it is also trivially true because 
$$|\text{supp}(m_{i,j,0})|\leq  |\{(\xi,\eta)\in \R^{2d}\colon |\xi|\leq 2^{i+1};|\eta|\leq 2^{j+1}\}|\lesssim 2^{(i+j)d}.$$

Also define $m_{i,j}^{k}(\xi,\eta)=\hat{\mu}(\xi,\eta)\hat{\psi}(2^{-i}\xi)\hat{\psi}(2^{-j}\eta)\rho^{k}(\xi,\eta)$.

We recall that in \cite{triangleaveraging} they showed that for any multi-indices $\alpha,\beta$
\begin{equation}
|\partial_{\xi}^{\alpha}\partial_{\eta}^{\beta}\hat{\mu}(\xi,\eta)|\leq C_{\alpha,\beta} \left(1+\min\{|\xi|,|\eta|\}|\sin(\theta)|\right)^{-\frac{d-2}{2}} (1+|(\xi,\eta)|)^{-\frac{d-2}{2}}.
\end{equation}

Assume $i\geq j$. In this case 
$$\min\{|\xi|,|\eta|\}|\sin(\theta)|\sim 2^{j-k}.$$

As in \cite{triangleaveraging}, when $k\leq \lfloor{j/2}\rfloor$, the derivatives of the cutoff functions $\rho_k$ are all bounded, and so are the derivatives of $\rho^{\lfloor{j/2}\rfloor}$. Hence, for any multi-indices $\alpha,\beta$
\begin{equation}
|\partial_{\xi}^{\alpha}\partial_{\eta}^{\beta}m_{i,j,k}|\leq C_{\alpha,\beta} 2^{-(j-k)\frac{d-2}{2}} 2^{-i\frac{d-2}{2}}
\end{equation}
for all $(\xi,\eta)\in\text{supp}(m_{i,j,k}) $ and 
\begin{equation}
|\partial_{\xi}^{\alpha}\partial_{\eta}^{\beta}m_{i,j}^{\lfloor j/2\rfloor}|\leq C_{\alpha,\beta}  2^{-i\frac{d-2}{2}}
\end{equation}
for all $(\xi,\eta)\in\text{supp}(m_{i,j}^{\lfloor j/2\rfloor}) $.

An application of Corollary \ref{corl2l2l1} will give us the control of $\|T_{m_{i,j,k}}\|_{L^2\times L^2 \rightarrow L^1}$ for $k\leq \lfloor j/2 \rfloor$
\begin{equation}
\begin{split}
     \|\mathcal{T}_{1}^{i,j,k}\|_{L^2\times L^2\rightarrow L^1}\lesssim &2^{-(j-k)\frac{d-2}{2}} 2^{-i\frac{d-2}{2}} |\text{supp}(m_{i,j,k})|^{1/4}\\
     \lesssim & 2^{-(i+j)\frac{d-2}{2}} 2^{k\frac{d-2}{2}}2^{(i+j)\frac{d}{4}}2^{-k\frac{d-1}{4}}\\
     \lesssim & 2^{-(i+j)\frac{d-4}{4}} 2^{k\frac{d-3}{4}}.
\end{split}
\end{equation}

For $\|T_{m_{i,j}^{\lfloor j/2 \rfloor}}\|_{L^2\times L^2\rightarrow L^1}$, Theorem \ref{ghscriteria} will give us
\begin{equation}
%\begin{split}
\|T_{m_{i,j}^{\lfloor j/2 \rfloor}}\|_{L^2\times L^2\rightarrow L^1}\lesssim  (2^{-\frac{i(d-2)}{2}})^{\frac{3}{4}} \|(m_{i,j}^{\lfloor j/2 \rfloor})\|_{1}^{1/4}.
%\end{split}
\end{equation}

One can check that 
\begin{equation}
\begin{split}
    \|(m_{i,j}^{\lfloor j/2 \rfloor})\|_{1}\lesssim & 2^{-i(\frac{d-2}{2})}\int_{\text{supp}(m_{i,j}^{\lfloor j/2 \rfloor})} (1+2^{j}|\sin(\theta(\xi,\eta))|)^{-\frac{d-2}{2}}d\xi d\eta\\
    \lesssim &2^{-i(\frac{d-2}{2})}|\{|\xi|\sim 2^i,|\eta|\sim 2^j, |\sin(\theta)|\leq 2^{-j}\}|\\
    &+2^{-i(\frac{d-2}{2})}\int_{|\xi|\sim 2^i,\,|\eta|\sim 2^j, 2^{-j}\leq |\sin(\theta)|\leq 2^{-j/2}} (1+2^{j}|\sin(\theta(\xi,\eta))|)^{-\frac{d-2}{2}}d\xi d\eta\\
    \lesssim &2^{-i(\frac{d-2}{2})} \left(2^{(i+j)d}2^{-j(d-1)}+2^{-j(\frac{d-2}{2})}2^{(i+j)d}2^{-jd/4}\right)\\
    \lesssim &2^{-i(\frac{d-2}{2})}2^{-j(\frac{3d-4}{4})} 2^{(i+j)d}.
\end{split}
\end{equation}

Then 
$$\|T_{m_{i,j}^{\lfloor j/2\rfloor}}\|_{L^2\times L^2\rightarrow L^1}\lesssim 2^{-i(\frac{d-2}{2})}2^{(i+j)\frac{d}{4}}2^{-j(\frac{3d-4}{16})}.$$

Hence, for $i\geq j$
\begin{equation}
\begin{split} \|\mathcal{T}_{1}^{i,j}\|_{L^2\times L^2\rightarrow L^1}\leq &\sum_{k\geq 0}^{\lfloor j/2\rfloor-1} \|\mathcal{T}_{1}^{i,j,k}\|_{L^2\times L^2\rightarrow L^1} +\|\mathcal{T}_{m_{i,j}^{\lfloor j/2\rfloor}}\|_{L^2\times L^2\rightarrow L^1} \\
\lesssim&\left(\sum_{k=0}^{\lfloor j/2\rfloor-1} 2^{-(i+j)\frac{d-4}{4}}2^{k\frac{d-3}{4}}\right)+2^{-i(\frac{d-2}{2})}2^{(i+j)\frac{d}{4}}2^{-j(\frac{3d-4}{16})}\\
\lesssim & 2^{-(i+j)\frac{d-4}{4}}2^{\frac{j}{2}\frac{d-3}{4}}+2^{-i(\frac{d-4}{4})}2^{j(\frac{d+4}{16})}\\
\lesssim &2^{-i\frac{d-4}{4}}2^{-j(\frac{d-5}{8})}+2^{-i(\frac{d-4}{4})}2^{j(\frac{d+4}{16})}\lesssim 2^{-i(\frac{d-4}{4})}2^{j(\frac{d+4}{16})}.\\
\end{split}
\end{equation}

When $j\geq i$ one can get instead, 
$$\|\mathcal{T}_1^{i,j}\|_{L^{2}\times L^2\rightarrow L^1} \lesssim 2^{-j(\frac{d-4}{4})}2^{i(\frac{d+4}{4})}.$$

In any case,
\begin{equation}
    \begin{split}
\|\mathcal{T}_1^{i,j}\|_{L^{2}\times L^2\rightarrow L^1} \lesssim& 2^{-\max\{i,j\}(\frac{d-4}{4})}2^{\min\{i,j\}(\frac{d+4}{16})}\\
\leq& 2^{-\frac{(i+j)}{2}(\frac{d-4}{4})}2^{\frac{(i+j)}{2}\frac{d+4}{16}}=2^{-\frac{(i+j)}{2}(\frac{3d-20}{16})}
    \end{split}
\end{equation}
which finishes the proof of the proposition with $\delta=\frac{3d-20}{32}$.
\end{proof}

\bibliographystyle{alpha}
\bibliography{sources}

\newcommand{\etalchar}[1]{$^{#1}$}
\begin{thebibliography}{BGH{\etalchar{+}}18}

\bibitem[BFO{\etalchar{+}}23]{BFOPZ}
Tainara Borges, Benjamin Foster, Yumeng Ou, Jill Pipher, and Zirui Zhou.
\newblock Sparse bounds for the bilinear spherical maximal function.
\newblock {\em Journal of the London Mathematical Society}, 107(4):1409--1449, (2023).

\bibitem[BGH{\etalchar{+}}18]{BGHHO}
Jose Barrionuevo, Loukas Grafakos, Danqing He, Petr Honz\'{\i}k, and Lucas Oliveira.
\newblock Bilinear spherical maximal function.
\newblock {\em Math. Res. Lett.}, 25(5):1369--1388, (2018).

\bibitem[Bou85]{bourgaind2}
Jean Bourgain.
\newblock Estimations de certaines fonctions maximales.
\newblock {\em C. R. Acad. Sci. Paris S\'{e}r. I Math.}, 301(10):499--502, (1985).

\bibitem[BS98]{BS}
Jong-Guk Bak and Yong-Sun Shim.
\newblock Endpoint inequalities for spherical multilinear convolutions.
\newblock {\em J. Funct. Anal.}, 157(2):534--553, (1998).

\bibitem[Cal79]{calderon}
Calixto Calder\'{o}n.
\newblock Lacunary spherical means.
\newblock {\em Illinois J. Math.}, 23(3):476--484, (1979).

\bibitem[CK17]{cladek2017}
Laura Cladek and Ben Krause.
\newblock Improved endpoint bounds for the lacunary spherical maximal operator.
\newblock {\em preprint arXiv:1703.01508, to appear in Anal. PDE}, (2017).

\bibitem[CW78]{coifmanweiss}
Ronald Coifman and Guido Weiss.
\newblock {Review: R. E. Edwards and G. I. Gaudry, Littlewood-Paley and multiplier theory}.
\newblock {\em Bulletin of the American Mathematical Society}, 84(2):242 -- 250, (1978).

\bibitem[CZ22]{MCZZ}
Michael Christ and Zirui Zhou.
\newblock A class of singular bilinear maximal functions.
\newblock {\em preprint arXiv:2203.16725}, (2022).

\bibitem[DR22]{dosidisramos}
Georgios Dosidis and João P.~G. Ramos.
\newblock The multilinear spherical maximal function in one dimension, (2022).

\bibitem[DV98]{DuoanVargas}
Javier Duoandikoetxea and Ana Vargas.
\newblock Maximal operators associated to {F}ourier multipliers with an arbitrary set of parameters.
\newblock {\em Proc. Roy. Soc. Edinburgh Sect. A}, 128(4):683--696, (1998).

\bibitem[GGI{\etalchar{+}}13]{GGIS}
Dan-Andrei Geba, Allan Greenleaf, Alex Iosevich, Eyvindur Palsson, and Eric Sawyer.
\newblock Restricted convolution inequalities, multilinear operators and applications.
\newblock {\em Math. Res. Lett.}, 20(4):675--694, (2013).

\bibitem[GGIP15]{GGIP}
Loukas Grafakos, Allan Greenleaf, Alex Iosevich, and Eyvindur Palsson.
\newblock Multilinear generalized {R}adon transforms and point configurations.
\newblock {\em Forum Math.}, 27(4):2323--2360, (2015).

\bibitem[GHH21]{GHH}
Loukas Grafakos, Danqing He, and Petr Honz\'{\i}k.
\newblock Maximal operators associated with bilinear multipliers of limited decay.
\newblock {\em J. Anal. Math.}, 143(1):231--251, (2021).

\bibitem[GHS20]{GHS}
Loukas Grafakos, Danqing He, and Lenka Slav\'{\i}kov\'{a}.
\newblock {$L^2\times L^2\to L^1$} boundedness criteria.
\newblock {\em Math. Ann.}, 376(1-2):431--455, (2020).

\bibitem[GI12]{threepointconf}
Allan Greenleaf and Alex Iosevich.
\newblock On triangles determined by subsets of the {E}uclidean plane, the associated bilinear operators and applications to discrete geometry.
\newblock {\em Anal. PDE}, 5(2):397--409, (2012).

\bibitem[GIKL22]{radonplane}
A.~Greenleaf, A.~Iosevich, B.~Krause, and A.~Liu.
\newblock {$L^p$} estimates for bilinear generalized {R}adon transforms in the plane.
\newblock In {\em Combinatorial and {A}dditive {N}umber {T}heory {V}}, volume 395 of {\em Springer Proc. Math. Stat.}, pages 179--198. Springer, Cham, (2022).

\bibitem[Gra14]{classicalgrafakos}
Loukas Grafakos.
\newblock {\em Classical {F}ourier analysis}, volume 249 of {\em Graduate Texts in Mathematics}.
\newblock Springer, New York, third edition, (2014).

\bibitem[HHY20]{HHY}
Yaryong Heo, Sunggeum Hong, and Chan~Woo Yang.
\newblock Improved bounds for the bilinear spherical maximal operators.
\newblock {\em Math. Res. Lett.}, 27(2):397--434, (2020).

\bibitem[IL19]{IL}
Alex Iosevich and Bochen Liu.
\newblock Equilateral triangles in subsets of {$\mathbb{R}^d$} of large {H}ausdorff dimension.
\newblock {\em Israel J. Math.}, 231(1):123--137, (2019).

\bibitem[IPS22]{IPS}
Alex Iosevich, Eyvindur Palsson, and Sean Sovine.
\newblock Simplex averaging operators: quasi-{B}anach and {$L^p$}-improving bounds in lower dimensions.
\newblock {\em J. Geom. Anal.}, 32(3):Paper No. 87, 16, (2022).

\bibitem[JL20]{JL}
Eunhee Jeong and Sanghyuk Lee.
\newblock Maximal estimates for the bilinear spherical averages and the bilinear {B}ochner-{R}iesz operators.
\newblock {\em J. Funct. Anal.}, 279(7):108629, 29, (2020).

\bibitem[LS23]{leeshuin}
Sanghyuk Lee and Kalachand Shuin.
\newblock Bilinear maximal functions associated with degenerate surfaces.
\newblock {\em J. Funct. Anal.}, 285(8):Paper No. 110070, 26, (2023).

\bibitem[Obe88]{DOberlin}
Daniel Oberlin.
\newblock Multilinear convolutions defined by measures on spheres.
\newblock {\em Trans. Amer. Math. Soc.}, 310(2):821--835, (1988).

\bibitem[PS20]{triangleaveraging}
Eyvindur Palsson and Sean Sovine.
\newblock The triangle averaging operator.
\newblock {\em J. Funct. Anal.}, 279(8):108671, 21, (2020).

\bibitem[PS22]{sparsetriangle}
Eyvindur Palsson and Sean Sovine.
\newblock Sparse bounds for maximal triangle and bilinear spherical averaging operators.
\newblock {\em preprint arXiv:2110.08928}, (2022).

\bibitem[SS21]{SS}
Saurabh Shrivastava and Kalachand Shuin.
\newblock ${L}^p$ estimates for multilinear convolution operators defined with spherical measure.
\newblock {\em Bulletin of the London Mathematical Society}, 53(4):1045–1060, (2021).

\bibitem[Ste76]{stein}
Elias Stein.
\newblock Maximal functions. {I}. {S}pherical means.
\newblock {\em Proc. Nat. Acad. Sci. U.S.A.}, 73(7):2174--2175, (1976).

\bibitem[SWW95]{SeegerWainger}
Andreas Seeger, Stephen Wainger, and James Wright.
\newblock Pointwise convergence of spherical means.
\newblock {\em Math. Proc. Cambridge Philos. Soc.}, 118(1):115--124, (1995).

\end{thebibliography}

\end{document}